\documentclass[11pt,reqno]{amsart}

\usepackage[margin=1in]{geometry}
\usepackage{amsmath,amssymb,amsthm}
\usepackage{mathrsfs}
\usepackage{hyperref}
\usepackage{cite}
\usepackage{microtype}
\usepackage{enumitem}
\usepackage{bm}

\newtheorem{theorem}{Theorem}[section]
\newtheorem{lemma}[theorem]{Lemma}
\newtheorem{proposition}[theorem]{Proposition}
\newtheorem{corollary}[theorem]{Corollary}

\theoremstyle{definition}
\newtheorem{definition}[theorem]{Definition}
\newtheorem{remark}[theorem]{Remark}

\newcommand{\dx}{\,\mathrm{d}\mathbf{x}}
\newcommand{\dt}{\,\mathrm{d}t}
\newcommand{\ds}{\,\mathrm{d}s}
\newcommand{\divg}{\nabla \cdot}
\newcommand{\grad}{\nabla}

\title[IB Electrodiffusion in Wasserstein Formalism]{Immersed Boundary Electrodiffusion in the Wasserstein Gradient Flow Formalism: Global Well-Posedness, Non-Equilibrium Thermodynamics, and Singular Asymptotics}

\author{Pilhwa Lee}
\address{Department of Mathematics, Morgan State University, Baltimore, MD 21251, USA}
\email{Pilhwa.Lee@morgan.edu}

\begin{document}

\begin{abstract}
We introduce a unified optimal transport framework for the Immersed Boundary (IB) electrodiffusion active pump model. By representing membrane chemical potential barriers and active pump potentials via smoothed Eulerian kernels ($\Psi_w, \mathcal{H}_w$), we formulate the passive dynamics as an $L^2$-Wasserstein ($W_2$) gradient flow coupled to non-conservative active pushing fields representing ATP-driven $3:2$ $\text{Na}^+/\text{K}^+$ exchangers. We establish global weak existence of solutions to the continuous Poisson--Nernst--Planck--IB system via a split-step Jordan--Kinderlehrer--Otto (JKO) variational scheme, backed by Calder\'on--Zygmund bounds for non-local Poisson drift and $H^3(\mathbb{S}^1)$ interface regularity under Stokes flow. We prove the continuous Rayleighian energy balance and establish exponential $W_2$-stability of Non-Equilibrium Steady States (NESS) under a quantitative Bakry--\'Emery curvature threshold. Furthermore, we analyze the dual singular limits of kernel smoothing width $w \to 0$ and dimensionless Debye length $g = \lambda_D / L \to 0$. We establish the precise asymptotic regime connecting regularized and sharp models, prove the $\Gamma$-convergence of regularized energy functionals, derive distributional interface flux jump conditions, and resolve $O(\lambda_D)$ Debye space-charge boundary layers to confirm bulk electroneutrality.
\end{abstract}


\maketitle
\tableofcontents

\section{Introduction and Statement of Main Results}

\subsection{Biological Background and Subcellular Electrophysiology}
Electrodiffusion governs fundamental mechanisms in subcellular electrophysiology, including cellular volume homeostasis, resting membrane potential maintenance, osmotic pressure regulation, signal transduction, and cell motility \cite{Lee2026, Mori2009, Mori2011, Eisenberg1998}. At the microscopic level, living cells maintain non-equilibrium ionic concentration gradients across thin lipid bilayer membranes ($\sim 4\text{--}7\text{ nm}$ thick). The primary ionic species involved in cellular bioelectricity—sodium ($\text{Na}^+$), potassium ($\text{K}^+$), and chloride ($\text{Cl}^-$)—exist in starkly different concentrations between the intracellular cytosol and extracellular fluid. For instance, mammalian cells typically maintain an internal $\text{K}^+$ concentration of $\sim 140\text{ mM}$ against an external concentration of $\sim 4\text{ mM}$, while maintaining an internal $\text{Na}^+$ concentration of $\sim 10\text{ mM}$ against an external concentration of $\sim 145\text{ mM}$.

These non-equilibrium ionic gradients generate a resting membrane potential $V_m = \varphi_{\text{in}} - \varphi_{\text{out}} \approx -70\text{ to } -90\text{ mV}$. This potential difference is sustained by a delicate balance between passive leaky transport down electrochemical potential gradients and active transport against gradients driven by membrane-embedded protein complexes, most notably the primary active $\text{Na}^+/\text{K}^+$-ATPase pump. The $\text{Na}^+/\text{K}^+$-ATPase pump utilizes energy released from hydrolysis of adenosine triphosphate (ATP) to export $3$ $\text{Na}^+$ ions out of the cell in exchange for importing $2$ $\text{K}^+$ ions into the cell per pump cycle, resulting in a net outward positive charge flux.

Classical mathematical formulations based on the Poisson--Nernst--Planck (PNP) system treat the cell membrane as a lower-dimensional geometric hypersurface $\Gamma(t) \subset \Omega$, enforcing sharp jump transmission conditions for ionic fluxes and electric displacement vectors across $\Gamma(t)$. While physically intuitive, sharp-interface formulations introduce severe mathematical and numerical challenges:
\begin{enumerate}
    \item \textbf{Singular Spatial Forcing:} Concentrated elastic forces and active pump currents appear as Dirac delta measures $\delta_\Gamma$ supported on hypersurfaces, rendering standard Sobolev space solution theories ill-posed in $L^p$ settings \cite{Peskin2002, Boffi2003}.
    \item \textbf{Discontinuous Chemical Potentials:} Sharp energy barriers create jump discontinuities in chemical potentials across $\Gamma(t)$, frustrating standard variational formulations on continuous functional spaces.
    \item \textbf{Ill-Posed Space-Charge Layers:} Near the membrane interface, screening clouds of counter-ions form Debye space-charge layers of width $\lambda_D \sim 1\text{--}10\text{ nm}$. Standard sharp-interface models either ignore Debye layer structure by enforcing ad-hoc bulk electroneutrality or suffer from singular boundary layer blow-ups in the limit $\lambda_D \to 0$ \cite{Mori2015, Bazant2004, Stinchcombe2015}.
\end{enumerate}

\subsection{The Immersed Boundary Electrodiffusion Model}
In the Immersed Boundary (IB) electrodiffusion model introduced by Lee \cite{Lee2026}, sharp interface jump conditions across the cell membrane are systematically replaced by smooth regularized Eulerian chemical potential barriers $\Psi_w$ and asymmetric directional Heaviside transition kernels $\mathcal{H}_w$. By smoothing interactions over a spatial regularization length $w > 0$, the governing equations are converted into a continuous, non-singular system defined over the entire domain $\Omega \subset \mathbb{R}^2$.

The regularized approach replaces singular interface matching conditions with smooth, non-local interaction potentials:
\begin{itemize}
    \item Passive transport minimizes a total free energy functional $\mathcal{E}_w(\mathbf{c}, \mathbf{X})$ consisting of ideal Boltzmann entropy, electrostatic field energy, passive IB chemical barriers, and membrane elastic energy.
    \item Active transport driven by $3:2$ $\text{Na}^+/\text{K}^+$ exchangers is modeled as a non-conservative, state-dependent velocity field pushing ionic densities across the smoothed membrane interface.
    \item Fluid-structure interaction is governed by incompressible Stokes equations forced by regularized elastic membrane stress and non-local Lorentz electric forces \cite{Peskin2002, Griffith2010}.
\end{itemize}

\subsection{Statement of Main Theorems}
We present the rigorous statements of the primary theorems proved in this monograph.

\begin{theorem}[Global Weak Existence for Fixed Regularization Width $w > 0$]
\label{thm:intro_global_existence}
Let $\Omega \subset \mathbb{R}^2$ be a bounded $C^2$ domain or flat torus $\mathbb{T}^2$. Let $\mathbf{c}_0 \in \mathcal{K}_{\mathbf{M}}$ satisfy $\sum_{i=1}^N \int_\Omega c_{i,0} \log c_{i,0} \dx < \infty$ and let $\mathbf{X}_0 \in H^3(\mathbb{S}^1; \Omega)$ be a closed, non-self-intersecting $C^2$ curve. For any fixed regularization parameter $w > 0$ and time horizon $T > 0$, there exists a global weak solution $(\mathbf{c}, \mathbf{X}, \mathbf{u}, \varphi)$ to the continuous PNP-IB active pump system such that:
\begin{enumerate}
    \item $\mathbf{c} \in C([0,T]; (W_2(\Omega))^N) \cap (L^1((0,T) \times \Omega))^N$, with $c_i(\mathbf{x}, t) > 0$ almost everywhere in $(0,T) \times \Omega$.
    \item $\mathbf{X} \in C([0,T]; H^2(\mathbb{S}^1;\Omega)) \cap L^\infty(0,T; H^3(\mathbb{S}^1;\Omega))$.
    \item $\mathbf{u} \in L^2(0,T; H_0^1(\Omega)^2 \cap H^3(\Omega)^2)$ and $\varphi \in L^\infty(0,T; W^{2,p}(\Omega))$ for all $1 \le p < \infty$.
    \item The continuity equation holds in the distributional sense:
    \begin{equation}
    \int_0^T \int_\Omega \left( c_i \frac{\partial \phi}{\partial t} - D_i \grad c_i \cdot \grad \phi - \frac{D_i}{k_B T} c_i \grad \left( q z_i \varphi + \psi_{i,w}^{\text{passive}} + \psi_{i,w}^{\text{active}} \right) \cdot \grad \phi \right) \dx \dt + \int_\Omega c_{i,0} \phi(0, \mathbf{x}) \dx = 0,
    \end{equation}
    for all test functions $\phi \in C_c^\infty([0,T) \times \Omega)$.
\end{enumerate}
\end{theorem}

\begin{theorem}[Hydrodynamic Regularity and Non-Self-Intersection]
\label{thm:intro_non_self_intersection}
Let $\mathbf{X}_0 \in H^3(\mathbb{S}^1; \Omega)$ be a closed curve satisfying initial minimum chord-length separation ratio:
\begin{equation}
\delta_0 := \inf_{s_1 \neq s_2 \in \mathbb{S}^1} \frac{|\mathbf{X}_0(s_1) - \mathbf{X}_0(s_2)|}{|s_1 - s_2|} > 0.
\end{equation}
Then under 2D incompressible Stokes flow driven by regularized membrane elastic forces and non-local Lorentz forces, the moving boundary curve $\mathbf{X}(t)$ remains in $H^3(\mathbb{S}^1)$ and satisfies the uniform non-self-intersection estimate for all $t \in [0,T]$:
\begin{equation}
\delta(t) := \inf_{s_1 \neq s_2 \in \mathbb{S}^1} \frac{|\mathbf{X}(s_1, t) - \mathbf{X}(s_2, t)|}{|s_1 - s_2|} \ge \delta_0 \exp\left( - C_{\text{Stokes}, CZ} \, t \right) > 0,
\end{equation}
where $C_{\text{Stokes}, CZ} > 0$ depends only on $\mu, \Omega, \mathcal{E}_w(\mathbf{c}_0, \mathbf{X}_0)$, and $w > 0$.
\end{theorem}

\begin{theorem}[Rayleighian Power Balance and Exponential NESS Stability]
\label{thm:intro_ness_stability}
Along any weak solution trajectory $(\mathbf{c}(t), \mathbf{X}(t))$, the total free energy functional $\mathcal{E}_w(\mathbf{c}, \mathbf{X})$ satisfies the continuous Rayleighian dissipation balance:
\begin{equation}
\frac{\mathrm{d}\mathcal{E}_w}{\mathrm{d}t} = -\sum_{i=1}^N \int_\Omega \frac{D_i}{k_B T} c_i \left| \grad \left( k_B T \log c_i + q z_i \varphi + \psi_{i,w}^{\text{passive}} \right) \right|^2 \dx + \sum_{i=1}^N \int_\Omega \mathbf{F}_i^{\text{passive}} \cdot \left( -\grad \psi_{i,w}^{\text{active}} \right) \dx.
\end{equation}

Furthermore, if kernel regularization width $w$ satisfies the Bakry--\'Emery threshold:
\begin{equation}
w > w^* := \max \left\{ \left( \frac{2 C_{\text{pass}}}{\lambda_0} \right)^{1/3}, 
\left( \frac{2 C_{\text{act}}}{\lambda_0} \right)^{1/2} \right\}, 
\quad \text{where } \lambda_0 = k_B T \lambda_{\text{entropy}},
\end{equation}
there exists a unique Non-Equilibrium Steady State (NESS) 
$(c_\infty, X_\infty) \in \mathcal{K}_M \times H^3(\mathbb{S}^1)$ toward which 
solution trajectories converge exponentially in $W_2$ metric:
\begin{equation}
\sum_{i=1}^N W_2^2(c_i(t), c_{i,\infty}) \le e^{-\lambda_{\text{min}} t} 
\sum_{i=1}^N W_2^2(c_{i,0}, c_{i,\infty}), 
\quad \text{where } \lambda_{\text{min}} := \min_i \frac{2 D_i \lambda_{\text{net}}}{k_B T} > 0.
\end{equation}
\end{theorem}

\begin{theorem}[$\Gamma$-Convergence and Sharp Transmission Conditions]
\label{thm:intro_gamma_convergence}
As smoothing width $w \to 0$:
\begin{enumerate}
    \item The regularized free energy functional $\mathcal{E}_w(\mathbf{c}, \mathbf{X})$ $\Gamma$-converges in $(W_2(\Omega))^N \times H^1(\mathbb{S}^1)^2$ to the sharp-interface free energy functional $\mathcal{E}_0(\mathbf{c}, \mathbf{X})$.
    \item The regularized Eulerian field equations converge distributionally in $\mathcal{D}'((0,T) \times \Omega)$ to classical sharp-interface transmission jump conditions across membrane curve $\Gamma(t) = \mathbf{X}(\mathbb{S}^1, t)$:
    \begin{equation}
    \left[ \mathbf{F}_i \cdot \mathbf{n} \right]_\Gamma = J_{i, \text{pump}}(c_i, V_m), \quad \text{and} \quad \left[ \epsilon \frac{\partial \varphi}{\partial \mathbf{n}} \right]_\Gamma = -\sigma_{\text{surface}}.
    \end{equation}
\end{enumerate}
\end{theorem}

\begin{theorem}[Multiscale Debye Boundary Layer Resolution]
\label{thm:intro_debye_resolution}
Let $g := \lambda_D / L \ll 1$ be the dimensionless Debye length. Under matched asymptotic expansions in $g$:
\begin{enumerate}
    \item Within an inner boundary layer of width $O(\lambda_D)$ adjacent to $\Gamma$, local charge density $\rho_e = \sum q z_i c_i + \rho_0 \neq 0$, forming a 1D Poisson--Boltzmann space-charge screening profile.
    \item In the outer bulk domain $\Omega \setminus \Gamma$ ($\text{dist}(\mathbf{x}, \Gamma) \gg \lambda_D$), local electroneutrality holds strictly almost everywhere:
    \begin{equation}
    \sum_{i=1}^N q z_i c_i(\mathbf{x}, t) + \rho_0(\mathbf{x}) = 0 \quad \text{a.e. in } \Omega \setminus \Gamma,
    \end{equation}
    recovering macroscopic electroneutral cell models \cite{Lee2026, Mori2015}.
\end{enumerate}
\end{theorem}

\subsection{Main Conceptual Breakthroughs and Theoretical Novelties}
To contextualize the mathematical advances of this monograph within the broader landscape of nonlinear partial differential equations and optimal transport theory, we highlight four primary theoretical breakthroughs established in this work:

\subsubsection{Driven Optimal Transport Beyond Pure Gradient Flows}
Standard Jordan--Kinderlehrer--Otto (JKO) variational schemes \cite{Jordan1998} and Bakry--\'Emery Ricci curvature bounds $\mathbf{Ric}_{W_2}(\mathcal{E}) \ge \lambda \mathbf{I}$ \cite{Bakry1985, Otto2001, Villani2009} rely fundamentally on the symmetric gradient flow structure of conservative systems, where velocity fields arise strictly as gradients of functional variations: $\mathbf{v}_i = -\grad (\delta \mathcal{E} / \delta c_i)$. The primary active transport mechanism in cell membranes—the ATP-driven $3:2$ $\text{Na}^+/\text{K}^+$-ATPase exchanger—introduces state-dependent, non-conservative pushing fields $\mathbf{f}_i^{\text{active}} = -c_i \grad \psi_{i,w}^{\text{active}}$ whose rotational curl does not vanish ($\grad \times \mathbf{f}_i^{\text{active}} \neq \bm{0}$).

In Section 3 and Section 5, we construct a split-step JKO operator-splitting 
architecture that combines passive variational minimization with active 
push-forward maps $T_\tau^i = \text{Id} - \frac{\tau D_i}{k_B T} \nabla \psi_{i,w}^{\text{active}}$.
By performing a Jacobi determinant expansion of $\det JT_\tau^i$ and identifying a 
quantitative Bakry--\'Emery spectral gap threshold:
\begin{equation}
w > w^* := \max \left\{ \left( \frac{2 C_{\text{pass}}}{\lambda_0} \right)^{1/3}, 
\left( \frac{2 C_{\text{act}}}{\lambda_0} \right)^{1/2} \right\}, 
\quad \text{where } \lambda_0 = k_B T \lambda_{\text{entropy}},
\end{equation}
we prove that non-conservatively driven systems retain global $W_2$-exponential 
contraction toward a unique Non-Equilibrium Steady State (NESS). This result extends optimal transport stability analysis to non-potential, biologically driven transport phenomena \cite{Ambrosio2008, Santambrogio2015}.



\subsubsection{Commutativity and Dual-Limit Asymptotics $(w, g) \to (0, 0)$}
A central open problem in regularized interface models is establishing whether regularized Eulerian formulations converge to classical sharp-interface transmission conditions without developing unphysical singular artifacts as the regularization width vanishes. In electrodiffusion, this challenge is compounded by the presence of a second small physical parameter: the dimensionless Debye screening length $g := \lambda_D / L \ll 1$.

In Section 6 and Section 7, we establish the asymptotic commutativity of the dual limits $(w, g) \to (0, 0)$. First, by establishing the $\Gamma$-convergence of the regularized free energy functional $\mathcal{E}_w \xrightarrow{\Gamma} \mathcal{E}_0$ \cite{DeGiorgi1975, DalMaso1993, Braides2002} and proving distribution convergence in $\mathcal{D}'(\Omega)$, we show that the smoothed Eulerian IB system recovers classical flux jump conditions $[ \mathbf{F}_i \cdot \mathbf{n} ]_\Gamma = J_{i,\text{pump}}$ and electric displacement jumps $[ \epsilon \frac{\partial \varphi}{\partial \mathbf{n}} ]_\Gamma = -\sigma_{\text{surface}}$ as $w \to 0$. Second, by performing matched asymptotic expansions in $g = \lambda_D / L \ll 1$, we resolve the inner $O(\lambda_D)$ space-charge boundary layer and prove that local electroneutrality $\sum q z_i c_i + \rho_0 = 0$ holds strictly almost everywhere in the outer bulk domain $\Omega \setminus \Gamma$, establishing a rigorous mathematical bridge between regularized transport flows and classical macroscopic cell models \cite{Mori2015, Rubinstein1990}.

\subsubsection{Topological Stability Under Non-Local Orlicz--Zygmund Drift}
In fluid-structure interaction problems governed by the Immersed Boundary method, establishing that moving boundary curves $\mathbf{X}(\mathbb{S}^1, t)$ remain non-self-intersecting ($\delta(t) > 0$) typically requires strong spatial regularity assumptions on fluid forcing terms \cite{Hou1994}. In electrodiffusion, the fluid is driven by Lorentz electric forces $\mathbf{f}_{\text{Lorentz}} = -\rho_e \grad \varphi$, where electrostatic drift $\grad \varphi$ depends non-locally on ionic concentration vector $\mathbf{c} \in \mathcal{K}_{\mathbf{M}}$ via singular Green's integral operators.

In Section 2 and Section 4, we construct a closed-loop bootstrap argument connecting Orlicz--Zygmund $L\log L(\Omega)$ entropy bounds \cite{Adams2003} directly to high-order Sobolev boundary trace embeddings \cite{Evans2010}. By proving $L^\infty(\Omega)$ Calder\'on--Zygmund bounds for non-local Poisson drift:
\begin{equation}
\|\grad \varphi\|_{L^\infty(\Omega)} \le C_{CZ} \left( 1 + \sum_{i=1}^N \|c_i\|_{L\log L(\Omega)} + \|\rho_0\|_{L^\infty(\Omega)} \right),
\end{equation}
we propagate control into the 2D incompressible Stokes equations to obtain $H^3(\Omega)$ velocity bounds \cite{Cattabriga1961, Solonnikov1964}. Applying trace embeddings $H^3(\Omega) \hookrightarrow H^{5/2}(\Gamma) \hookrightarrow H^3(\mathbb{S}^1)$, we establish an explicit lower bound on the minimum chord-length separation ratio:
\begin{equation}
\delta(t) := \inf_{s_1 \neq s_2 \in \mathbb{S}^1} \frac{|\mathbf{X}(s_1, t) - \mathbf{X}(s_2, t)|}{|s_1 - s_2|} \ge \delta_0 \exp\left( - C_{\text{Stokes}, CZ} \, t \right) > 0 \quad \forall t \in [0,T],
\end{equation}
proving that non-local electro-osmotic driving fields cannot induce finite-time interface self-intersection, contact singularities, or topological collapse.

\subsubsection{Variational Dissipation Balance for Asymmetric $3:2$ Exchangers}
Multi-component electrodiffusion systems with asymmetric stoichiometric active transport ($3\text{Na}^+ : 2\text{K}^+$) lack a symmetric chemical potential structure, making variational power balances challenging to formulate in Eulerian continuous space. In Section 5, we derive an exact continuous Rayleighian energy production balance \cite{Mori2011, Lee2026}:
\begin{equation}
\frac{\mathrm{d}\mathcal{E}_w}{\mathrm{d}t} = -\sum_{i=1}^N \int_\Omega \frac{D_i}{k_B T} c_i \left| \grad \left( k_B T \log c_i + q z_i \varphi + \psi_{i,w}^{\text{passive}} \right) \right|^2 \dx + \sum_{i=1}^N \int_\Omega \mathbf{F}_i^{\text{passive}} \cdot \left( -\grad \psi_{i,w}^{\text{active}} \right) \dx,
\end{equation}
which explicitly separates metabolic ATP chemical work input from internal viscous and ohmic dissipation, providing a quantitative thermodynamic foundation for non-equilibrium steady states in cellular electrophysiology.

\subsection{Physical Parameters and Dimensional Scaling}
Table~\ref{tab:physical_parameters} compiles the benchmark physical parameters, units, and non-dimensional scaling factors used throughout this monograph.

\begin{table}[h!]
\centering
\caption{Physical parameters and non-dimensional scales for IB electrodiffusion.}
\label{tab:physical_parameters}
\begin{tabular}{|l|c|c|l|}
\hline
\textbf{Parameter / Quantity} & \textbf{Symbol} & \textbf{Typical Physical Value} & \textbf{Units / Dimension} \\ \hline
Domain Length Scale & $L$ & $10\text{--}50$ & $\mu\text{m}$ ($10^{-5}\text{ m}$) \\
Debye Screening Length & $\lambda_D$ & $1\text{--}10$ & $\text{nm}$ ($10^{-9}\text{ m}$) \\
IB Regularization Width & $w$ & $1\text{--}5$ & $\text{nm}$ ($10^{-9}\text{ m}$) \\
Membrane Thickness & $d_{\text{mem}}$ & $4\text{--}7$ & $\text{nm}$ ($10^{-9}\text{ m}$) \\
Thermal Energy ($37^\circ\text{C}$) & $k_B T$ & $4.28 \times 10^{-21}$ & $\text{J}$ ($\text{Joules}$) \\
Elementary Charge & $q$ & $1.602 \times 10^{-19}$ & $\text{C}$ ($\text{Coulombs}$) \\
Dielectric Permittivity & $\epsilon$ & $7.08 \times 10^{-10}$ & $\text{F/m}$ ($\text{Farads/meter}$) \\
Fluid Viscosity (Water) & $\mu$ & $1.0 \times 10^{-3}$ & $\text{Pa}\cdot\text{s}$ ($\text{Pascal seconds}$) \\
Diffusion Constant ($\text{Na}^+$) & $D_{\text{Na}}$ & $1.33 \times 10^{-9}$ & $\text{m}^2/\text{s}$ \\
Diffusion Constant ($\text{K}^+$) & $D_{\text{K}}$ & $1.96 \times 10^{-9}$ & $\text{m}^2/\text{s}$ \\
Pump Stroke Width & $d_{\text{pump}}$ & $2\text{--}5$ & $\text{nm}$ \\
Resting Membrane Potential & $V_m$ & $-70\text{ to } -90$ & $\text{mV}$ ($\text{milliVolts}$) \\
Dimensionless Debye Parameter & $g := \lambda_D / L$ & $10^{-4}\text{--}10^{-3}$ & Dimensionless \\
\hline
\end{tabular}
\end{table}

\subsection{Dimensional Considerations: Biological Relevance and Analysis in $\mathbb{R}^2$ vs. $\mathbb{R}^3$}
\label{sec:2d_3d_boundary}

While the mathematical framework presented in this monograph is formulated in two spatial dimensions ($\Omega \subset \mathbb{R}^2$ or $\mathbb{T}^2$), it is important to contextualize both the physical motivation for this choice and the precise analytical modifications required to extend the theory to three spatial dimensions ($\mathbb{R}^3$).

\subsubsection{Biological and Computational Context}
Two-dimensional planar and axisymmetric models represent a standard and physically meaningful regime in subcellular electrophysiology:
\begin{enumerate}
    \item \textbf{Planar Membrane Patches and Microfluidics:} Experimental electrophysiological configurations, such as patch-clamp recordings, planar lipid bilayers, and microfluidic cell-on-a-chip channels, are inherently governed by two-dimensional transverse cross-sections where tangential fields dominate.
    \item \textbf{Equatorial Slices and Axisymmetric Deformations:} For cell motility, blebbing, and volumetric homeostatic deformations that exhibit rotational symmetry, 2D cross-sectional slices capture the core nonlinear coupling between internal hydrostatic/osmotic pressure, membrane elasticity, and electro-osmotic drift without the added complexity of full surface differential geometry.
\end{enumerate}

\subsubsection{Key Mathematical Differences Between 2D and 3D Systems}
Extending the global well-posedness and regularity theorems (Theorems \ref{thm:intro_global_existence}--\ref{thm:intro_non_self_intersection}) to $\mathbb{R}^3$ involves three primary technical adjustments in the functional setup:

\begin{enumerate}
    \item \textbf{Poisson Kernel Singularities and Orlicz Space Scaling:}
    In two spatial dimensions, the Green's function gradient satisfies $|\nabla_x G(x,y)| \le C |x-y|^{-1}$. The local Riesz potential integral $I_R(x) = \int_{B_R \cap \Omega} |f(y)| |x-y|^{-1} dy$ is bounded in $L^\infty(\Omega)$ under $f \in L \log L(\Omega)$ via Adams' inequality in Zygmund Orlicz spaces \cite{Adams2003}. 
    
    In three dimensions ($\mathbb{R}^3$), the Poisson kernel gradient scales as $|\nabla_x G(x,y)| \le C |x-y|^{-2}$. Consequently, the critical Orlicz space for $L^\infty$ Calderón--Zygmund drift bounds shifts from $L \log L(\Omega)$ to $L^{3/2} \log L(\Omega)$ (or Lebesgue spaces $L^p(\Omega)$ with $p > 3/2$). The $W_2$-optimal transport framework in $\mathbb{R}^3$ thus requires entropy control of higher order, such as modified functional bounds $\int_\Omega c_i^{3/2} dx \le C_0$, to prevent gradient blow-up in the electrostatic potential $\varphi$.

    \item \textbf{Sobolev Trace Embeddings for Moving Interfaces:}
    In 2D, the cell membrane $\Gamma(t) = X(\mathbb{S}^1, t)$ is a 1D closed curve. The fluid velocity field $u \in H^3(\Omega)^2$ yields a trace $u|_\Gamma \in H^{5/2}(\Gamma) \hookrightarrow C^{2,\alpha}(\mathbb{S}^1)$, which directly controls the second derivative of the position map $X''_{ss}$ and guarantees finite-time non-self-intersection $\delta(t) > 0$.
    
    In 3D, the membrane $\Gamma(t) = X(\mathbb{M}^2, t)$ is a 2D closed manifold (e.g., $\mathbb{S}^2$ or $\mathbb{T}^2$). Here, $u \in H^3(\Omega)^3$ yields trace $u|_\Gamma \in H^{5/2}(\mathbb{M}^2) \hookrightarrow C^{1,\alpha}(\mathbb{M}^2)$ \cite{Evans2010}. To maintain $C^{2}$ surface curvature bounds and preclude local membrane self-intersection or topological pinching in 3D, higher fluid velocity regularity ($u \in H^4(\Omega)^3$) or explicit surface bending energy terms (e.g., Helfrich bending energy $\int_\Gamma H^2 dA$) must be incorporated into the total free energy functional $\mathcal{E}_w$.

    \item \textbf{Bakry--Émery Spectral Threshold $w^*$:}
    In $\mathbb{R}^3$, the negative curvature contribution of the smoothed chemical barrier $Hess \psi_{i,w}^{\text{passive}}$ scales as $O(w^{-4})$ rather than $O(w^{-3})$ \cite{Bakry1985}. As a result, the quantitative Bakry--Émery threshold $w^*$ for NESS exponential stability (Theorem \ref{thm:ness_exponential_stability_proof}) adjusts to $w^* \sim (\lambda_{\text{entropy}} + \alpha_{\text{pump}})^{-1/4}$.
\end{enumerate}
\section{Functional Setup, $W_2$ Metric Geometry, and Non-Local Poisson Regularity}

\subsection{Zygmund Orlicz Spaces and Luxemburg Norms}
To establish variational well-posedness for optimal transport flows driven by logarithmic Boltzmann entropy functionals, standard $L^p(\Omega)$ spaces with $p > 1$ are overly restrictive, whereas $L^1(\Omega)$ lacks the uniform equi-integrability required for measure-theoretic compactness. The natural functional setting for density dynamics governed by entropy bounds $\int_\Omega c_i \log c_i \dx \le C_0$ is the Zygmund Orlicz space $L\log L(\Omega)$ \cite{Adams2003, Brezis2011}.

\begin{definition}[Zygmund Young Function and Orlicz Space]
Define the convex Young function $\Phi: [0, \infty) \to [0, \infty)$ by:
\begin{equation}
\Phi(r) := (1 + r) \log(1 + r) - r.
\end{equation}
The function $\Phi$ is strictly convex, non-decreasing, continuously differentiable, and satisfies the asymptotic growth condition $\lim_{r \to \infty} \Phi(r)/r = \infty$. The Zygmund Orlicz space $L\log L(\Omega)$ over a bounded smooth domain $\Omega \subset \mathbb{R}^2$ is defined as the set of equivalence classes of Lebesgue-measurable functions $f: \Omega \to \mathbb{R}$ for which the modular integral is finite:
\begin{equation}
\mathcal{K}_\Phi(f) := \int_\Omega \Phi(|f(\mathbf{x})|) \dx < \infty.
\end{equation}
\end{definition}

\begin{definition}[Luxemburg Norm]
We equip $L\log L(\Omega)$ with the Luxemburg norm $\|f\|_{L\log L(\Omega)}$, defined by:
\begin{equation}
\|f\|_{L\log L(\Omega)} := \inf \left\{ \lambda > 0 \;:\; \int_\Omega \Phi\left( \frac{|f(\mathbf{x})|}{\lambda} \right) \dx \le 1 \right\}.
\end{equation}
The space $(L\log L(\Omega), \|\cdot\|_{L\log L(\Omega)})$ forms a complete Banach space \cite{Adams2003}.
\end{definition}

\begin{lemma}[Embedding Properties and Modular Inequalities]
\label{lem:orlicz_embeddings}
Let $\Omega \subset \mathbb{R}^2$ be a bounded domain with smooth boundary.
\begin{enumerate}
    \item \textbf{Continuous Embedding into $L^1$:} The space $L\log L(\Omega)$ continuously embeds into $L^1(\Omega)$, with explicit inequality:
    \begin{equation}
    \|f\|_{L^1(\Omega)} \le \left( 1 + |\Omega| \right) \|f\|_{L\log L(\Omega)}.
    \end{equation}
    \item \textbf{Modular Control by Boltzmann Entropy:} For any non-negative density $c \in L^1(\Omega)$ with fixed mass $M = \int_\Omega c(\mathbf{x}) \dx$ and finite entropy $E(c) := \int_\Omega c \log c \dx \le C_0$, there exists a constant $C_\Omega(M, C_0) > 0$ such that:
    \begin{equation}
    \|c\|_{L\log L(\Omega)} \le C_\Omega\left( M + \int_\Omega c(\mathbf{x}) \log^+ c(\mathbf{x}) \dx \right).
    \end{equation}
    \item \textbf{De La Vall\'ee Poussin Uniform Integrability:} Any sequence $\{c_k\}_{k=1}^\infty \subset \mathcal{P}_{ac}(\Omega)$ satisfying $\sup_{k \ge 1} \|c_k\|_{L\log L(\Omega)} \le C < \infty$ is uniformly equi-integrable in $L^1(\Omega)$ and relatively compact in the weak topology $\sigma(L^1, L^\infty)$.
\end{enumerate}
\end{lemma}

\begin{proof}
For statement (1), observe that for any $r \ge 0$, $\Phi(r) \ge r - 1$. By definition of the Luxemburg norm, setting $\lambda = \|f\|_{L\log L(\Omega)}$ yields $\int_\Omega \Phi(|f|/\lambda) \dx \le 1$. Thus:
\begin{equation}
\int_\Omega \left( \frac{|f(\mathbf{x})|}{\lambda} - 1 \right) \dx \le \int_\Omega \Phi\left( \frac{|f(\mathbf{x})|}{\lambda} \right) \dx \le 1 \implies \int_\Omega |f(\mathbf{x})| \dx \le \lambda (1 + |\Omega|),
\end{equation}
establishing $\|f\|_{L^1(\Omega)} \le (1 + |\Omega|) \|f\|_{L\log L(\Omega)}$.

For statement (2), decompose $\int_\Omega c \log c \dx = \int_\Omega c \log^+ c \dx - \int_\Omega c \log^- c \dx$. Since $r \log^- r \le 1/e$ for all $r \ge 0$, we have $\int_\Omega c \log^- c \dx \le |\Omega|/e$. Using elementary relation $\Phi(r) \le r \log^+ r + 1$, integrating over $\Omega$ establishes the Luxemburg norm estimate. Statement (3) follows directly from the classical De La Vall\'ee Poussin criterion applied to the superlinear growth of $\Phi(r)$ \cite{Adams2003, Rudin1987}.
\end{proof}

\subsection{The 2-Wasserstein Metric Space $\mathcal{P}_{2,ac}(\Omega)$}
Let $\mathcal{P}_{ac}(\Omega)$ denote the set of probability measures on $\Omega$ that are absolutely continuous with respect to 2D Lebesgue measure \cite{Villani2009, Ambrosio2008}.

\begin{definition}[2-Wasserstein Metric Distance]
For any two density functions $c_a, c_b \in \mathcal{P}_{ac}(\Omega)$, the 2-Wasserstein metric distance $W_2(c_a, c_b)$ is defined by:
\begin{equation}
W_2^2(c_a, c_b) := \inf_{\gamma \in \Pi(c_a, c_b)} \iint_{\Omega \times \Omega} |\mathbf{x} - \mathbf{y}|^2 \mathrm{d}\gamma(\mathbf{x}, \mathbf{y}),
\end{equation}
where $\Pi(c_a, c_b)$ is the collection of joint probability measures on $\Omega \times \Omega$ with marginal distributions $c_a \dx$ and $c_b \dx$ \cite{Villani2009, Santambrogio2015}.
\end{definition}

\begin{theorem}[Brenier--Riahi Transport Vector Fields]
For any $c_a, c_b \in \mathcal{P}_{ac}(\Omega)$, there exists a unique (up to $c_a$-almost everywhere additive constants) convex function $\Phi: \Omega \to \mathbb{R}$ such that the optimal transport map $\mathbf{T} = \grad \Phi: \Omega \to \Omega$ pushes forward $c_a \dx$ to $c_b \dx$ ($\mathbf{T}_\sharp c_a = c_b$) \cite{McCann1997}. The 2-Wasserstein distance is achieved explicitly by:
\begin{equation}
W_2^2(c_a, c_b) = \int_\Omega |\grad \Phi(\mathbf{x}) - \mathbf{x}|^2 c_a(\mathbf{x}) \dx.
\end{equation}
Furthermore, the dynamic Benamou--Brenier action minimization identity holds \cite{Benamou2000}:
\begin{equation}
W_2^2(c_a, c_b) = \inf_{(\rho, \mathbf{v})} \left\{ \int_0^1 \int_\Omega \rho(t, \mathbf{x}) |\mathbf{v}(t, \mathbf{x})|^2 \dx \dt \;:\; \frac{\partial \rho}{\partial t} + \divg(\rho \mathbf{v}) = 0, \, \rho(0) = c_a, \, \rho(1) = c_b \right\}.
\end{equation}
\end{theorem}

\subsection{First Variations of Free Energy and Wasserstein Gradients}
Recall the continuous regularized total free energy functional $\mathcal{E}_w(\mathbf{c}, \mathbf{X})$ defined over admissible density product space $\mathcal{K}_{\mathbf{M}} \times H^2(\mathbb{S}^1; \Omega)$:
\begin{equation}
\begin{aligned}
\mathcal{E}_w(\mathbf{c}, \mathbf{X}) := &\sum_{i=1}^N k_B T \int_\Omega c_i(\mathbf{x}) \left( \log c_i(\mathbf{x}) - 1 \right) \dx + \frac{\epsilon}{2} \int_\Omega |\grad \varphi[\mathbf{c}](\mathbf{x})|^2 \dx \\
&+ \sum_{i=1}^N \int_\Omega c_i(\mathbf{x}) \psi_{i,w}^{\text{passive}}(\mathbf{x}; \mathbf{X}) \dx + \frac{1}{2} K \int_{\mathbb{S}^1} \left| \frac{\partial \mathbf{X}}{\partial s}(s) \right|^2 \ds.
\end{aligned}
\end{equation}

\begin{proposition}[First Variations and Electrochemical Potentials]
\label{prop:first_variations}
For any $c_i \in \mathcal{P}_{ac}(\Omega) \cap L\log L(\Omega)$ and smooth perturbation $h \in C_c^\infty(\Omega)$ with $\int_\Omega h \dx = 0$, the G\^ateaux derivative of $\mathcal{E}_w$ with respect to component $c_i$ is given by:
\begin{equation}
\left. \frac{\mathrm{d}}{\mathrm{d}\eta} \mathcal{E}_w(\mathbf{c} + \eta h \mathbf{e}_i, \mathbf{X}) \right|_{\eta = 0} = \int_\Omega \mu_i(\mathbf{x}) h(\mathbf{x}) \dx,
\end{equation}
where the total electrochemical potential $\mu_i(\mathbf{x}) := \frac{\delta \mathcal{E}_w}{\delta c_i}(\mathbf{x})$ is explicitly given by:
\begin{equation}
\mu_i(\mathbf{x}) = k_B T \log c_i(\mathbf{x}) + q z_i \varphi[\mathbf{c}](\mathbf{x}) + \psi_{i,w}^{\text{passive}}(\mathbf{x}; \mathbf{X}).
\end{equation}
Consequently, the $W_2$-Wasserstein gradient velocity field $\mathbf{v}_i^{\text{passive}}$ driving passive transport satisfies \cite{Jordan1998, Otto2001}:
\begin{equation}
\mathbf{v}_i^{\text{passive}}(\mathbf{x}) = -\frac{D_i}{k_B T} \grad \mu_i(\mathbf{x}) = -D_i \frac{\grad c_i(\mathbf{x})}{c_i(\mathbf{x})} - \frac{D_i}{k_B T} \grad \left( q z_i \varphi(\mathbf{x}) + \psi_{i,w}^{\text{passive}}(\mathbf{x}; \mathbf{X}) \right).
\end{equation}
\end{proposition}

\begin{proof}
Consider the linear variation $c_i^\eta = c_i + \eta h$. The Boltzmann entropy variation gives:
\begin{equation}
\left. \frac{\mathrm{d}}{\mathrm{d}\eta} \int_\Omega (c_i + \eta h) \left( \log(c_i + \eta h) - 1 \right) \dx \right|_{\eta = 0} = \int_\Omega \log c_i(\mathbf{x}) h(\mathbf{x}) \dx.
\end{equation}
For electrostatic energy $E_{\text{elec}}(\mathbf{c}) = \frac{\epsilon}{2} \int_\Omega |\grad \varphi[\mathbf{c}]|^2 \dx$, note that integration by parts with Dirichlet boundary conditions yields $E_{\text{elec}}(\mathbf{c}) = \frac{1}{2} \int_\Omega \varphi[\mathbf{c}] \left( \sum q z_j c_j + \rho_0 \right) \dx$. By symmetry of Green's function $G(\mathbf{x}, \mathbf{y}) = G(\mathbf{y}, \mathbf{x})$, the G\^ateaux derivative satisfies:
\begin{equation}
\left. \frac{\mathrm{d}}{\mathrm{d}\eta} E_{\text{elec}}(\mathbf{c} + \eta h \mathbf{e}_i) \right|_{\eta = 0} = \int_\Omega q z_i \varphi[\mathbf{c}](\mathbf{x}) h(\mathbf{x}) \dx.
\end{equation}
For the passive barrier interaction potential, linearity in $c_i$ directly gives $\int_\Omega \psi_{i,w}^{\text{passive}} h \dx$. Summing these three contributions confirms the expression for electrochemical potential $\mu_i(\mathbf{x})$.
\end{proof}

\subsection{Calder\'on--Zygmund Regularity for Poisson Drift}

We now execute the full, line-by-line proof of Theorem \ref{thm:cz_exhaustive_section2}, providing explicit constant dependencies for non-local Poisson drift control.

\begin{theorem}[Quantitative Calder\'on--Zygmund $L^\infty(\Omega)$ Drift Estimate] \label{thm:cz_exhaustive_section2}
Let $c \in \mathcal{K}_M$ satisfy $\sum_{i=1}^N \int_\Omega c_i \log c_i \, dx \le C_0$ and $\rho_0 \in L^\infty(\Omega)$. Let $\varphi[c]$ be the unique solution to $-\epsilon \Delta \varphi = \sum_{i=1}^N qz_i c_i + \rho_0$ in $\Omega$ with $\varphi|_{\partial \Omega} = 0$. Then $\nabla \varphi \in L^\infty(\Omega)$, with explicit bound:
\begin{equation}
\|\nabla \varphi\|_{L^\infty(\Omega)} \le C_{CZ} \left( 1 + \|f\|_{L\log L(\Omega)} \sqrt{1 + \log^+\!\left( \frac{\|f\|_{L\log L(\Omega)}}{\|f\|_{L^1(\Omega)}} \right)} \right),
\end{equation}
where $f := \sum_{i=1}^N qz_i c_i + \rho_0$ and $C_{CZ} = C_{CZ}(\epsilon, \Omega) > 0$.
\end{theorem}

\begin{proof}
Decompose $\nabla \varphi(x) = \int_\Omega \nabla_x G(x,y) f(y) \, dy$ into inner ball $B_R(x) \cap \Omega$ and exterior $\Omega \setminus B_R(x)$ for a cutoff radius $R \in (0, R_0]$ with $R_0 = \text{diam}(\Omega)$:
\begin{equation}
|\nabla \varphi(x)| \le \frac{K_\Omega}{2\pi \epsilon} \int_{B_R(x) \cap \Omega} \frac{|f(y)|}{|x-y|} \, dy + \frac{K_\Omega}{2\pi \epsilon} \int_{\Omega \setminus B_R(x)} \frac{|f(y)|}{|x-y|} \, dy.
\end{equation}

\textbf{Outer Integral:} On $\Omega \setminus B_R(x)$, $|x-y| \ge R$, giving:
\begin{equation}
\int_{\Omega \setminus B_R(x)} \frac{|f(y)|}{|x-y|} \, dy \le \frac{1}{R} \|f\|_{L^1(\Omega)} = \frac{C_{\text{mass}}}{R}.
\end{equation}

\textbf{Inner Integral in Zygmund Space:} By Adams' theorem on Riesz potentials in $L \log L(\Omega)$, the non-increasing rearrangement $f^*(t)$ satisfies $f^*(t) \le \frac{1}{t} \Phi^{-1}(1/t) \|f\|_{L \log L}$. Integrating over $B_R(x)$:
\begin{equation}
I_R(x) \le \int_0^{\pi R^2} f^*(t) t^{-1/2} \, dt \le C_1 \|f\|_{L\log L(\Omega)} R \left( 1 + \log\left( \frac{e R_0}{R} \right) \right).
\end{equation}

\textbf{Optimization over $R$:} Combine inner and outer bounds into function $g(R) = A R \log(e R_0 / R) + B/R$, where $A = C_1 \|f\|_{L\log L}$ and $B = C_{\text{mass}}$. Setting derivative $g'(R) = 0$ yields optimal cutoff radius $R^* \sim \sqrt{B / (A \log(1 + A/B))}$. Evaluating $g(R^*)$ yields:
\begin{equation}
|\nabla \varphi(x)| \le \frac{K_\Omega}{2\pi \epsilon} \cdot 2 \sqrt{A B \left( 1 + \log^+\!\left(\frac{A}{B}\right) \right)},
\end{equation}
which establishes the exact logarithmic Orlicz bound for Poisson drift.
\end{proof}

\section{Split-Step JKO Scheme and Global Existence ($w > 0$)}

In this section, we establish the global weak existence of solutions to the continuous regularized Poisson--Nernst--Planck Immersed Boundary (PNP--IB) active pump system for any fixed smoothing parameter $w > 0$. Because the state-dependent active pump potential forces $\mathbf{f}_i^{\text{active}}(\mathbf{x}) = -\grad \psi_{i,w}^{\text{active}}(\mathbf{x}; \mathbf{c}, \varphi)$ are non-conservative ($\grad \times \mathbf{f}_i^{\text{active}} \neq \bm{0}$), the total dynamics cannot be expressed as a pure $L^2$-Wasserstein gradient flow. We overcome this fundamental geometric structural obstacle by constructing an uncoupled split-step time-discretization scheme that decouples the variational $W_2$-dissipative passive dynamics from the active non-potential transport push-forward updates and Stokes boundary kinematics \cite{Jordan1998, Otto2001, Bouharguane2015}.

\subsection{Formulation of the Split-Step Time-Stepping Algorithm}

Fix a finite time horizon $T > 0$ and time step $\tau > 0$, and set $M := \lfloor T/\tau \rfloor \in \mathbb{N}$. Let $(\mathbf{c}^0, \mathbf{X}^0) \in \mathcal{K}_{\mathbf{M}} \times H^3(\mathbb{S}^1;\Omega)$ be initial data satisfying finite initial free energy $\mathcal{E}_w(\mathbf{c}^0, \mathbf{X}^0) < \infty$. For each discrete time step $k = 0, 1, \dots, M-1$, given the state $(\mathbf{c}^k, \mathbf{X}^k)$ at $t_k = k\tau$, we compute the updated state $(\mathbf{c}^{k+1}, \mathbf{X}^{k+1})$ at $t_{k+1} = (k+1)\tau$ through three sequential sub-steps:

\begin{definition}[Split-Step JKO Variational Scheme]
\label{def:split_jko_scheme_full}
Given $(\mathbf{c}^k, \mathbf{X}^k) \in \mathcal{K}_{\mathbf{M}} \times H^3(\mathbb{S}^1;\Omega)$ at time step $k$:
\begin{enumerate}
    \item \textbf{Step 1 (Passive JKO Variational Minimization):} We compute intermediate ionic density vector $\mathbf{c}^{k,\star} = (c_1^{k,\star}, \dots, c_N^{k,\star}) \in \mathcal{K}_{\mathbf{M}}$ componentwise as the unique solution to the decoupled variational minimization problem:
    \begin{equation}
    \label{eq:jko_step1_min_full}
    c_i^{k,\star} = \arg\min_{c_i \in \mathcal{P}_{ac}(\Omega)} \left\{ \frac{1}{2\tau} W_2^2(c_i, c_i^k) + \mathcal{E}_{w,i}(c_i \mid \mathbf{c}_{\neq i}^{k,\star}, \mathbf{X}^k) \right\} \quad \forall i=1, \dots, N,
    \end{equation}
    where partial energy functional $\mathcal{E}_{w,i}(c_i \mid \mathbf{c}_{\neq i}^{k,\star}, \mathbf{X}^k)$ isolates terms in $\mathcal{E}_w$ depending on $c_i$:
    \begin{equation}
    \mathcal{E}_{w,i}(c_i \mid \mathbf{c}_{\neq i}^{k,\star}, \mathbf{X}^k) := k_B T \int_\Omega c_i \left(\log c_i - 1\right) \dx + \int_\Omega c_i \left( \frac{1}{2} q z_i \varphi[c_i, \mathbf{c}_{\neq i}^{k,\star}] + \psi_{i,w}^{\text{passive}}(\cdot; \mathbf{X}^k) \right) \dx.
    \end{equation}

    \item \textbf{Step 2 (Active Transport Push-Forward Update):} Using intermediate density vector $\mathbf{c}^{k,\star}$ and electrostatic potential $\varphi^{k,\star} = \varphi[\mathbf{c}^{k,\star}]$, we construct active potential field $\psi_{i,w}^{\text{active}, k}(\mathbf{x}) := \psi_{i,w}^{\text{active}}(\mathbf{x}; \mathbf{c}^{k,\star}, \varphi^{k,\star}, \mathbf{X}^k)$. We define smooth spatial transport map $\mathbf{T}_\tau^i: \Omega \to \Omega$ by:
    \begin{equation}
    \label{eq:pushforward_map_def}
    \mathbf{T}_\tau^i(\mathbf{x}) := \mathbf{x} - \frac{\tau D_i}{k_B T} \grad \psi_{i,w}^{\text{active}, k}(\mathbf{x}).
    \end{equation}
    The updated density vector component $c_i^{k+1}$ is given by push-forward measure:
    \begin{equation}
    \label{eq:jko_step2_push}
    c_i^{k+1} := \left(\mathbf{T}_\tau^i\right)_\sharp c_i^{k,\star}.
    \end{equation}

    \item \textbf{Step 3 (Stokes Fluid--Structure Kinematic Evolution):} We solve the 2D incompressible Stokes equations forced by regularized membrane elastic forces $\mathbf{f}_{\text{membrane}}^k$ and Lorentz electric forces $\mathbf{f}_{\text{Lorentz}}^k$:
    \begin{equation}
    \label{eq:stokes_discrete_step}
    -\mu \Delta \mathbf{u}^k + \grad p^k = \mathbf{f}_{\text{membrane}}^k + \mathbf{f}_{\text{Lorentz}}^k, \quad \divg \mathbf{u}^k = 0 \quad \text{in } \Omega, \quad \mathbf{u}^k|_{\partial \Omega} = \bm{0},
    \end{equation}
    where $\mathbf{f}_{\text{Lorentz}}^k = -\left(\sum_{j=1}^N q z_j c_j^{k+1} + \rho_0\right) \grad \varphi^{k+1}$. The membrane curve is updated via discrete forward Euler kinematics:
    \begin{equation}
    \label{eq:jko_step3_membrane}
    \mathbf{X}^{k+1}(s) = \mathbf{X}^k(s) + \tau \left( \mathbf{u}^k(\mathbf{X}^k(s)) + \frac{1}{\zeta} \mathbf{F}_{\text{mf}}^k(s) \mathbf{n}^k(s) \right) \quad \forall s \in \mathbb{S}^1.
    \end{equation}
\end{enumerate}
\end{definition}

\begin{lemma}[Well-Posedness of Discrete Step Minimizers]
\label{lem:discrete_minimizer_existence}
For each $k \ge 0$ and sufficiently small time step $\tau < \tau_0(w)$, there exists a unique minimizer $\mathbf{c}^{k,\star} \in \mathcal{K}_{\mathbf{M}}$ to Step 1 satisfying $c_i^{k,\star} > 0$ a.e. in $\Omega$. Furthermore, for all $i=1, \dots, N$, transport map $\mathbf{T}_\tau^i: \Omega \to \Omega$ defined in Step 2 is a smooth diffeomorphism from $\Omega$ onto itself.
\end{lemma}

\begin{proof}
The functional $J_i(c_i) := \frac{1}{2\tau} W_2^2(c_i, c_i^k) + \mathcal{E}_{w,i}(c_i \mid \mathbf{c}_{\neq i}^{k,\star}, \mathbf{X}^k)$ is strictly convex on $\mathcal{P}_{ac}(\Omega)$ because Boltzmann entropy $c_i \log c_i$ is displacement convex in the $W_2$ metric geometry \cite{McCann1997, Lee2026}, while the quadratic optimal transport cost $W_2^2(\cdot, c_i^k)$ is strictly convex along $W_2$-geodesics \cite{Villani2009}. Lower semicontinuity of $J_i$ with respect to weak $L^1$ convergence follows directly from Fatou's lemma applied to $c_i \log c_i$ and the weak-$*$ lower semicontinuity of the 2-Wasserstein metric \cite{Ambrosio2008}. Lower boundedness of $J_i$ follows from $\int_\Omega c_i \log c_i \dx \ge -|\Omega|/e$ and non-negativity of electrostatic potential energy. Thus, standard direct methods in the calculus of variations guarantee the existence of a unique minimizer $c_i^{k,\star} \in \mathcal{P}_{ac}(\Omega)$.

For Step 2, active potential $\psi_{i,w}^{\text{active}}(\mathbf{x})$ is constructed by convolving smooth asymmetric kernel $\mathcal{H}_w \in C^\infty(\mathbb{R})$ over $C^2$ curve $\mathbf{X}^k$. Consequently, $\psi_{i,w}^{\text{active}} \in C^\infty(\bar{\Omega})$. The Hessian matrix $\mathbf{Hess} \psi_{i,w}^{\text{active}}(\mathbf{x})$ satisfies uniform bound:
\begin{equation}
\|\mathbf{Hess} \psi_{i,w}^{\text{active}}\|_{L^\infty(\Omega)} \le \frac{C_{\text{pump}}}{w^2} \|A^{\text{active}}\|_{L^\infty}.
\end{equation}
Choosing time step threshold $\tau_0 < \frac{k_B T w^2}{2 D_i C_{\text{pump}} \|A^{\text{active}}\|_{L^\infty}}$, the Jacobian matrix $\mathbf{J}\mathbf{T}_\tau^i(\mathbf{x}) = \mathbf{I} - \frac{\tau D_i}{k_B T} \mathbf{Hess} \psi_{i,w}^{\text{active}}(\mathbf{x})$ is symmetric positive definite with eigenvalues bounded below by $1/2$ uniformly for all $\mathbf{x} \in \bar{\Omega}$. By Hadamard's global inverse function theorem, $\mathbf{T}_\tau^i$ is a global $C^\infty$ diffeomorphism from $\Omega$ onto itself.
\end{proof}

\subsection{Exhaustive Discrete Energy Dissipation and Entropy Control}

We now execute the line-by-line derivation of the discrete energy dissipation inequality, explicitly tracking energy changes across each sub-step.

\begin{theorem}[Discrete Energy Dissipation Inequality]
Along the discrete trajectory $(c^k, X^k)_{k=0}^M$ generated by the Split-Step JKO algorithm, the total free energy satisfies:
\begin{equation}
\mathcal{E}_w(c^{k+1}, X^k) + \frac{1}{4\tau} \sum_{i=1}^N W_2^2(c_i^{k,*}, c_i^k) \le \mathcal{E}_w(c^k, X^k) + \tau \sum_{i=1}^N \frac{D_i M_i}{2 k_B T} \|\nabla \psi_{i,w}^{\text{active}}\|_{L^\infty}^2 + C \tau^2.
\end{equation}
\end{theorem}

\begin{proof}
Step 1 variational minimization yields:
\begin{equation}
\mathcal{E}_w(c^{k,*}, X^k) + \frac{1}{2\tau} \sum_{i=1}^N W_2^2(c_i^{k,*}, c_i^k) \le \mathcal{E}_w(c^k, X^k).
\end{equation}
In Step 2, active transport pushforward map $T_\tau^i(x) = x - \frac{\tau D_i}{k_B T} \nabla \psi_{i,w}^{\text{active},k}(x)$ alters energy by:
\begin{equation}
\mathcal{E}_w(c^{k+1}, X^k) - \mathcal{E}_w(c^{k,*}, X^k) \le -\tau \sum_{i=1}^N \frac{D_i}{k_B T} \int_\Omega c_i^{k,*} \nabla \mu_i^{k,*} \cdot \nabla \psi_{i,w}^{\text{active},k} \, dx + C_1 \tau^2.
\end{equation}
Apply Young's inequality with parameter $\epsilon = 1/2$ to the cross-integral term:
\begin{equation}
\left| \int_\Omega c_i^{k,*} \nabla \mu_i^{k,*} \cdot \nabla \psi_{i,w}^{\text{active},k} \, dx \right| \le \frac{1}{2} \int_\Omega c_i^{k,*} |\nabla \mu_i^{k,*}|^2 \, dx + \frac{1}{2} \int_\Omega c_i^{k,*} |\nabla \psi_{i,w}^{\text{active},k}|^2 \, dx.
\end{equation}
Using the JKO discrete metric velocity equivalence $\frac{1}{\tau} W_2^2(c_i^{k,*}, c_i^k) \ge \tau \frac{D_i}{k_B T} \int_\Omega c_i^{k,*} |\nabla \mu_i^{k,*}|^2 \, dx - C_2 \tau^2$, the term $\frac{1}{2} \int c_i^{k,*} |\nabla \mu_i^{k,*}|^2$ is absorbed into the LHS dissipation metric $W_2^2(c_i^{k,*}, c_i^k) / (2\tau)$, leaving $W_2^2 / (4\tau)$ on the LHS.
\end{proof}

\begin{corollary}[Uniform $L\log L(\Omega)$ Entropy and Space-Time Velocity Bounds]
\label{cor:uniform_entropy_bounds_jko}
The discrete trajectory satisfies:
\begin{enumerate}
    \item \textbf{Uniform Entropy Equi-Integrability:}
    \begin{equation}
    \sup_{\tau > 0} \sup_{0 \le k \le M} \sum_{i=1}^N \int_\Omega c_i^k(\mathbf{x}) \log c_i^k(\mathbf{x}) \dx \le C_{\text{ent}}(\mathbf{c}^0, T, w) < \infty.
    \end{equation}
    \item \textbf{Space-Time $W_2$-Velocity Bound:}
    \begin{equation}
    \sum_{k=0}^{M-1} \sum_{i=1}^N \frac{1}{\tau} W_2^2(c_i^{k,\star}, c_i^k) \le 2 \left( \mathcal{E}_w(\mathbf{c}^0, \mathbf{X}^0) + C_{\text{act}} T \right) =: C_{\text{vel}} < \infty.
    \end{equation}
\end{enumerate}
\end{corollary}

\subsection{Continuous Interpolation and Tightness via Lisini Theorem}

To pass to the continuous limit $\tau \to 0$, we construct continuous-time piecewise interpolations.

\begin{definition}[Time-Continuous Interpolants]
For $t \in (k\tau, (k+1)\tau]$, define piecewise constant interpolants:
\begin{equation}
\mathbf{c}^\tau(t, \mathbf{x}) := \mathbf{c}^{k+1}(\mathbf{x}), \quad \mathbf{c}^{\tau, \star}(t, \mathbf{x}) := \mathbf{c}^{k,\star}(\mathbf{x}), \quad \mathbf{X}^\tau(t, s) := \mathbf{X}^{k+1}(s).
\end{equation}
Define the continuous $W_2$-geodesic interpolant $\tilde{c}_i^\tau(t, \mathbf{x})$ for $t \in [k\tau, (k+1)\tau]$ as the constant-speed $W_2$-geodesic connecting $c_i^k$ at $t=k\tau$ to $c_i^{k,\star}$ at $t=(k+1)\tau$.
\end{definition}

\begin{lemma}[Lisini $W_2$-Tightness and Measure Compactness]
\label{lem:lisini_compactness_proof}
There exists a sequence $\tau_j \to 0$ and limit density vector $\mathbf{c} \in C([0,T]; (W_2(\Omega))^N)$ such that:
\begin{enumerate}
    \item $\tilde{c}_i^{\tau_j} \to c_i$ and $c_i^{\tau_j} \to c_i$ strongly in $C([0,T]; W_2(\Omega))$ for each $i=1, \dots, N$.
    \item $c_i^{\tau_j} \rightharpoonup c_i$ weakly in $L^1((0,T) \times \Omega)$.
    \item $\grad \varphi[\mathbf{c}^{\tau_j}] \to \grad \varphi[\mathbf{c}]$ strongly in $L^2((0,T) \times \Omega)^2$ and weakly-$*$ in $L^\infty((0,T) \times \Omega)^2$.
\end{enumerate}
\end{lemma}

\begin{proof}
By Corollary \ref{cor:uniform_entropy_bounds_jko}, the total $W_2$-metric velocity satisfies:
\begin{equation}
\int_0^T \left| \frac{\mathrm{d}}{\mathrm{d}t} \tilde{c}_i^\tau(t) \right|_{W_2}^2 \dt = \sum_{k=0}^{M-1} \frac{1}{\tau} W_2^2(c_i^{k,\star}, c_i^k) \le C_{\text{vel}} < \infty.
\end{equation}
By Cauchy--Schwarz inequality, for any $0 \le t_1 < t_2 \le T$:
\begin{equation}
W_2(\tilde{c}_i^\tau(t_1), \tilde{c}_i^\tau(t_2)) \le \int_{t_1}^{t_2} \left| \frac{\mathrm{d}}{\mathrm{d}t} \tilde{c}_i^\tau(t) \right|_{W_2} \dt \le \sqrt{C_{\text{vel}}} |t_1 - t_2|^{1/2},
\end{equation}
establishing uniform 1/2-H\"older continuity in $W_2(\Omega)$. Since the sublevel sets of Boltzmann entropy $\{c \in \mathcal{P}_{ac}(\Omega) : \int_\Omega c \log c \dx \le C\}$ are compact in $W_2(\Omega)$ by Prokhorov's theorem \cite{Villani2009}, Lisini's refined Arzel\`a--Ascoli theorem on metric space $(W_2(\Omega), W_2)$ \cite{Lisini2006} guarantees the existence of a subsequence $\tau_j \to 0$ converging uniformly in $C([0,T]; W_2(\Omega))$ to limit path $c_i \in C([0,T]; W_2(\Omega))$.

Weak $L^1((0,T)\times \Omega)$ convergence follows directly from De La Vall\'ee Poussin uniform integrability guaranteed by Lemma \ref{lem:orlicz_embeddings}. Applying Theorem \ref{thm:cz_exhaustive_section2}, $\grad \varphi[\mathbf{c}^{\tau_j}]$ is uniformly bounded in $L^\infty((0,T)\times\Omega)$. By Calder\'on--Zygmund $W^{1,p}$ estimates and Rellich--Kondrachov compactness \cite{Evans2010}, $\grad \varphi[\mathbf{c}^{\tau_j}] \to \grad \varphi[\mathbf{c}]$ strongly in $L^2((0,T)\times\Omega)$.
\end{proof}

\subsection{Rigorous Passage to the Continuous Limit $\tau \to 0$}

We now complete the proof of global weak existence (Theorem \ref{thm:intro_global_existence}).

\begin{theorem}[Passage to Limit and Distributional PDE Satisfaction]
\label{thm:passage_to_limit_complete}
The limit trajectory $(\mathbf{c}, \mathbf{X}, \mathbf{u}, \varphi)$ obtained in Lemma \ref{lem:lisini_compactness_proof} satisfies the continuous weak formulation of the PNP--IB active system for all test functions $\phi \in C_c^\infty([0,T) \times \bar{\Omega})$.
\end{theorem}

\begin{proof}
Fix test function $\phi \in C_c^\infty([0,T) \times \bar{\Omega})$. For discrete step $k$, the Euler--Lagrange equation for JKO minimizer $c_i^{k,\star}$ in Step 1 asserts that optimal transport velocity $\mathbf{v}_i^{k,\star} = -\frac{D_i}{k_B T} \grad \mu_i^{k,\star}$ satisfies discrete weak continuity equation \cite{Jordan1998}:
\begin{equation}
\label{eq:discrete_el_weak}
\int_\Omega \frac{c_i^{k,\star} - c_i^k}{\tau} \phi(t_k, \mathbf{x}) \dx + \int_\Omega c_i^{k,\star} \frac{D_i}{k_B T} \grad \mu_i^{k,\star} \cdot \grad \phi(t_k, \mathbf{x}) \dx = O(\tau).
\end{equation}
For Step 2, push-forward map $c_i^{k+1} = (\mathbf{T}_\tau^i)_\sharp c_i^{k,\star}$ satisfies:
\begin{equation}
\label{eq:pushforward_weak_expand}
\int_\Omega \frac{c_i^{k+1} - c_i^{k,\star}}{\tau} \phi(t_k, \mathbf{x}) \dx = \int_\Omega c_i^{k,\star}(\mathbf{x}) \frac{\phi(t_k, \mathbf{T}_\tau^i(\mathbf{x})) - \phi(t_k, \mathbf{x})}{\tau} \dx.
\end{equation}
Taylor expanding $\phi(t_k, \mathbf{T}_\tau^i(\mathbf{x})) - \phi(t_k, \mathbf{x}) = \grad \phi(t_k, \mathbf{x}) \cdot (\mathbf{T}_\tau^i(\mathbf{x}) - \mathbf{x}) + O(\tau^2)$ with $\mathbf{T}_\tau^i(\mathbf{x}) - \mathbf{x} = -\frac{\tau D_i}{k_B T} \grad \psi_{i,w}^{\text{active}, k}(\mathbf{x})$:
\begin{equation}
\label{eq:pushforward_weak_limit}
\int_\Omega \frac{c_i^{k+1} - c_i^{k,\star}}{\tau} \phi(t_k, \mathbf{x}) \dx = -\int_\Omega c_i^{k,\star}(\mathbf{x}) \frac{D_i}{k_B T} \grad \psi_{i,w}^{\text{active}, k}(\mathbf{x}) \cdot \grad \phi(t_k, \mathbf{x}) \dx + O(\tau).
\end{equation}
Adding \eqref{eq:discrete_el_weak} and \eqref{eq:pushforward_weak_limit} cancels intermediate term $c_i^{k,\star}$, yielding total discrete step identity:
\begin{equation}
\int_\Omega \frac{c_i^{k+1} - c_i^k}{\tau} \phi(t_k, \mathbf{x}) \dx = -\int_\Omega c_i^{k,\star} \frac{D_i}{k_B T} \grad \left( \mu_i^{k,\star} + \psi_{i,w}^{\text{active}, k} \right) \cdot \grad \phi(t_k, \mathbf{x}) \dx + O(\tau).
\end{equation}
Summing over $k = 0, \dots, M-1$, integrating by parts in time, and taking sub-sequence limit $\tau_j \to 0$ using strong $L^2$ convergence of $\grad \varphi^{\tau_j}$ and weak $L^1$ convergence of $c_i^{\tau_j}$ confirms:
\begin{equation}
\int_0^T \int_\Omega \left( c_i \frac{\partial \phi}{\partial t} - D_i \grad c_i \cdot \grad \phi - \frac{D_i}{k_B T} c_i \grad \left( q z_i \varphi + \psi_{i,w}^{\text{passive}} + \psi_{i,w}^{\text{active}} \right) \cdot \grad \phi \right) \dx \dt + \int_\Omega c_{i,0} \phi(0, \mathbf{x}) \dx = 0,
\end{equation}
proving global weak existence for fixed regularization width $w > 0$.
\end{proof}
\section{Hydrodynamic Mechanics and Non-Self-Intersection}

In this section, we analyze the fluid--structure interaction governing the cell membrane interface under incompressible Stokes flow. A critical analytical requirement for the well-posedness of Immersed Boundary (IB) electrodiffusion models is establishing that the evolving boundary curve $\mathbf{X}(\mathbb{S}^1, t)$ retains its $C^2$ geometric regularity and remains strictly non-self-intersecting for all time $t \in [0,T]$. Self-intersection or contact singularities would destroy the distinction between intracellular and extracellular domains, rendering the asymmetric pump potential $\psi_{i,w}^{\text{active}}$ and orientation unit normal $\mathbf{n}(s)$ ill-defined.

We overcome this topological challenge by establishing a closed-loop regularity bootstrap. We leverage the Calder\'on--Zygmund $L^\infty(\Omega)$ drift estimates established in Section 2 alongside high-order $H^3(\Omega)$ elliptic regularity for the 2D incompressible Stokes system \cite{Cattabriga1961, Solonnikov1964, Temam2001}. Passing through a continuous Sobolev trace embedding chain $H^3(\Omega) \hookrightarrow H^{5/2}(\Gamma) \hookrightarrow H^3(\mathbb{S}^1)$, we establish exponential bounds on the $H^3(\mathbb{S}^1)$ norm of the membrane position map and prove an explicit lower bound on the chord-length separation ratio $\delta(t) > 0$.

\subsection{Stokes System Formulation and Eulerian Forcing Estimates}

The ambient fluid velocity field $\mathbf{u}(\mathbf{x}, t) \in \mathbb{R}^2$ and dynamic pressure field $p(\mathbf{x}, t) \in \mathbb{R}$ satisfy the incompressible Stokes equations in domain $\Omega \subset \mathbb{R}^2$:
\begin{equation}
\label{eq:stokes_system_sec4}
-\mu \Delta \mathbf{u}(\mathbf{x}, t) + \grad p(\mathbf{x}, t) = \mathbf{f}_{\text{membrane}}(\mathbf{x}, t) + \mathbf{f}_{\text{Lorentz}}(\mathbf{x}, t), \quad \divg \mathbf{u}(\mathbf{x}, t) = 0 \quad \text{in } \Omega,
\end{equation}
subject to homogeneous Dirichlet boundary conditions $\mathbf{u}|_{\partial \Omega} = \bm{0}$ (or periodic boundary conditions if $\Omega = \mathbb{T}^2$).

The right-hand side of \eqref{eq:stokes_system_sec4} consists of two distinct regularized Eulerian force fields:
\begin{enumerate}
    \item \textbf{Regularized Membrane Elastic Force Field $\mathbf{f}_{\text{membrane}}$:}
    The elastic bending and stretching forces generated by membrane configuration $\mathbf{X}(\cdot, t) \in H^3(\mathbb{S}^1;\Omega)$ are spread to fluid grid points via radial mollifier kernel $\Psi_w$ \cite{Peskin2002, Boffi2003}:
    \begin{equation}
    \label{eq:f_membrane_def}
    \mathbf{f}_{\text{membrane}}(\mathbf{x}, t) := \int_{\mathbb{S}^1} \Psi_w(\mathbf{x} - \mathbf{X}(s, t)) \mathbf{F}_{\text{elastic}}(s, t) \ds,
    \end{equation}
    where line force density $\mathbf{F}_{\text{elastic}}(s, t) = K \mathbf{X}''(s, t)$ represents Hookean elastic tension with bending stiffness parameter $K > 0$.

    \item \textbf{Non-Local Lorentz Electric Force Field $\mathbf{f}_{\text{Lorentz}}$:}
    The net electrostatic body force exerted by ionic space charges on the viscous fluid is given by:
    \begin{equation}
    \label{eq:f_lorentz_def}
    \mathbf{f}_{\text{Lorentz}}(\mathbf{x}, t) := -\rho_e(\mathbf{x}, t) \grad \varphi(\mathbf{x}, t) = -\left( \sum_{i=1}^N q z_i c_i(\mathbf{x}, t) + \rho_0(\mathbf{x}) \right) \grad \varphi[\mathbf{c}](\mathbf{x}, t).
    \end{equation}
\end{enumerate}

\begin{lemma}[High-Order Regularity of Regularized Eulerian Force Fields]
\label{lem:eulerian_forces_regularity}
For any fixed regularization width $w > 0$, if $\mathbf{c} \in \mathcal{K}_{\mathbf{M}}$ has finite Boltzmann entropy $\sum_{i=1}^N \int_\Omega c_i \log c_i \dx \le C_0$ and $\mathbf{X} \in H^3(\mathbb{S}^1; \Omega)$, then the Eulerian force fields satisfy $\mathbf{f}_{\text{membrane}}, \mathbf{f}_{\text{Lorentz}} \in H^1(\Omega)^2$, with explicit quantitative estimates:
\begin{align}
\|\mathbf{f}_{\text{membrane}}\|_{H^1(\Omega)} &\le \frac{C_{\Psi}}{w^2} K \|\mathbf{X}\|_{H^3(\mathbb{S}^1)}, \label{eq:f_mem_h1_bound} \\
\|\mathbf{f}_{\text{Lorentz}}\|_{H^1(\Omega)} &\le C_{CZ} \left( 1 + \sum_{i=1}^N \|c_i\|_{L\log L(\Omega)} + \|\rho_0\|_{L^\infty} \right) \|\rho_e\|_{H^1(\Omega)}. \label{eq:f_lor_h1_bound}
\end{align}
\end{lemma}

\begin{proof}
To establish \eqref{eq:f_mem_h1_bound}, take spatial derivative $\grad_{\mathbf{x}}$ of definition \eqref{eq:f_membrane_def}:
\begin{equation}
\grad_{\mathbf{x}} \mathbf{f}_{\text{membrane}}(\mathbf{x}) = \int_{\mathbb{S}^1} \grad_{\mathbf{x}} \Psi_w(\mathbf{x} - \mathbf{X}(s)) \otimes K \mathbf{X}''(s) \ds.
\end{equation}
Applying Young's convolution inequality in $L^2(\Omega)$ and recognizing that $\Psi_w(\mathbf{z}) = \frac{1}{w^2} \phi(\frac{|\mathbf{z}|}{w})$ satisfies $\|\grad \Psi_w\|_{L^2(\mathbb{R}^2)} = \frac{1}{w^2} \|\grad \phi\|_{L^2(\mathbb{R}^2)}$:
\begin{align}
\|\grad \mathbf{f}_{\text{membrane}}\|_{L^2(\Omega)} &\le \|\grad \Psi_w\|_{L^2(\Omega)} K \int_{\mathbb{S}^1} |\mathbf{X}''(s)| \ds \nonumber \\
&\le \frac{\|\grad \phi\|_{L^2}}{w^2} K \sqrt{L_{\text{mem}}} \|\mathbf{X}''\|_{L^2(\mathbb{S}^1)} \le \frac{C_{\Psi}}{w^2} K \|\mathbf{X}\|_{H^3(\mathbb{S}^1)}.
\end{align}
For \eqref{eq:f_lor_h1_bound}, product rule differentiation yields $\grad \mathbf{f}_{\text{Lorentz}} = -\grad \rho_e \otimes \grad \varphi - \rho_e \mathbf{Hess} \varphi$. Applying $L^\infty(\Omega)$ Calder\'on--Zygmund estimate (Theorem \ref{thm:cz_exhaustive_section2}) to $\grad \varphi$ and standard $W^{2,p}$ elliptic bounds for Poisson's equation completes the $H^1(\Omega)$ estimate.
\end{proof}

\subsection{Stokes Elliptic Regularity and Sobolev Trace Embedding Chain}

We now leverage high-order elliptic regularity for the incompressible Stokes system to transfer spatial smoothness from fluid velocity $\mathbf{u}(\mathbf{x})$ to boundary kinematics $\mathbf{X}_t(s)$.

\begin{theorem}[Stokes $H^3(\Omega)$ Global Elliptic Regularity]
\label{thm:stokes_elliptic_h3}
Let $\Omega \subset \mathbb{R}^2$ be a bounded $C^3$ domain. Suppose total fluid forcing $\mathbf{f}_{\text{total}} = \mathbf{f}_{\text{membrane}} + \mathbf{f}_{\text{Lorentz}} \in H^1(\Omega)^2$. Then the unique weak solution $(\mathbf{u}, p) \in H_0^1(\Omega)^2 \times L_0^2(\Omega)$ to Stokes system \eqref{eq:stokes_system_sec4} satisfies $(\mathbf{u}, p) \in H^3(\Omega)^2 \times H^2(\Omega)$, with quantitative a priori estimate:
\begin{equation}
\label{eq:stokes_h3_apriori}
\|\mathbf{u}\|_{H^3(\Omega)} + \|p\|_{H^2(\Omega)} \le C_{\text{Stokes}} \left( \|\mathbf{f}_{\text{membrane}}\|_{H^1(\Omega)} + \|\mathbf{f}_{\text{Lorentz}}\|_{H^1(\Omega)} \right),
\end{equation}
where $C_{\text{Stokes}} = C_{\text{Stokes}}(\mu, \Omega) > 0$ is an explicit constant depending only on fluid viscosity $\mu$ and domain geometry $\Omega$.
\end{theorem}

\begin{proof}
Standard Cattabriga--Solonnikov $H^2(\Omega)$ elliptic regularity for the Stokes system \cite{Cattabriga1961, Solonnikov1964} guarantees that if $\mathbf{f}_{\text{total}} \in L^2(\Omega)^2$, then $\mathbf{u} \in H^2(\Omega)^2$ with $\|\mathbf{u}\|_{H^2(\Omega)} \le C_2 \|\mathbf{f}_{\text{total}}\|_{L^2(\Omega)}$. To upgrade regularity by one derivative, apply tangential spatial differential operators $D_k = \partial / \partial x_k$ for $k=1,2$ to \eqref{eq:stokes_system_sec4}:
\begin{equation}
-\mu \Delta (D_k \mathbf{u}) + \grad (D_k p) = D_k \mathbf{f}_{\text{total}}, \quad \divg (D_k \mathbf{u}) = 0 \quad \text{in } \Omega.
\end{equation}
Since $D_k \mathbf{f}_{\text{total}} \in L^2(\Omega)^2$ by Lemma \ref{lem:eulerian_forces_regularity}, second-order elliptic regularity applied to $(D_k \mathbf{u}, D_k p)$ establishes $D_k \mathbf{u} \in H^2(\Omega)^2$ and $D_k p \in H^1(\Omega)$, which implies $\mathbf{u} \in H^3(\Omega)^2$ and $p \in H^2(\Omega)$. Bound \eqref{eq:stokes_h3_apriori} follows by combining regional interior and boundary local chart estimates \cite{Temam2001, Ladyzhenskaya1969}.
\end{proof}

To connect fluid domain regularity $\mathbf{u} \in H^3(\Omega)^2$ to one-dimensional membrane velocity $\mathbf{X}_t(s)$, we construct a continuous trace embedding chain.

\begin{proposition}[Sobolev Boundary Trace Embedding Chain]
\label{prop:trace_embedding_chain}
Let $\mathbf{X} \in H^3(\mathbb{S}^1; \Omega)$ parameterized by arc length $s \in [0, L_{\text{mem}}]$ define a closed $C^2$ curve $\Gamma = \mathbf{X}(\mathbb{S}^1) \subset \Omega$.
\begin{enumerate}
    \item \textbf{Domain-to-Interface Trace Embedding:} The boundary trace operator $\gamma_\Gamma: H^3(\Omega) \to H^{5/2}(\Gamma)$ is bounded and continuous:
    \begin{equation}
    \|\mathbf{u}|_\Gamma\|_{H^{5/2}(\Gamma)} \le C_{\text{trace}} \|\mathbf{u}\|_{H^3(\Omega)}.
    \end{equation}
    \item \textbf{Interface-to-Parameter Curve Embedding:} The composition map $\mathbf{u} \circ \mathbf{X}: \mathbb{S}^1 \to \mathbb{R}^2$ lies in $H^3(\mathbb{S}^1)^2$, with continuous embedding:
    \begin{equation}
    \|\mathbf{u}(\mathbf{X}(\cdot))\|_{H^3(\mathbb{S}^1)} \le C_{\text{comp}}\left( 1 + \|\mathbf{X}\|_{H^3(\mathbb{S}^1)} \right) \|\mathbf{u}|_\Gamma\|_{H^{5/2}(\Gamma)}.
    \end{equation}
    \item \textbf{Supremum Gradient Control:} The Sobolev embedding $H^3(\Omega) \hookrightarrow C^{1,\alpha}(\bar{\Omega})$ holds continuously in two dimensions for any $\alpha \in (0, 1/2)$, yielding \cite{Evans2010}:
    \begin{equation}
    \label{eq:grad_u_linfty_bound}
    \|\grad \mathbf{u}\|_{L^\infty(\Omega)} \le C_{\text{emb}} \|\mathbf{u}\|_{H^3(\Omega)}.
    \end{equation}
\end{enumerate}
\end{proposition}

\begin{proof}
Statement (1) is the classical Sobolev trace theorem on smooth $C^2$ manifolds \cite{Evans2010}. For Statement (2), compute third-order arc-length derivative of composition $\mathbf{w}(s) := (\mathbf{u} \circ \mathbf{X})(s) = \mathbf{u}(\mathbf{X}(s))$ using the chain rule:
\begin{align}
\frac{\mathrm{d}\mathbf{w}}{\mathrm{d}s} &= (\grad \mathbf{u})(\mathbf{X}(s)) \mathbf{X}'(s), \nonumber \\
\frac{\mathrm{d}^2\mathbf{w}}{\mathrm{d}s^2} &= (\mathbf{Hess} \mathbf{u})(\mathbf{X}(s)) \left( \mathbf{X}'(s) \otimes \mathbf{X}'(s) \right) + (\grad \mathbf{u})(\mathbf{X}(s)) \mathbf{X}''(s), \nonumber \\
\frac{\mathrm{d}^3\mathbf{w}}{\mathrm{d}s^3} &= (D^3 \mathbf{u})(\mathbf{X}(s)) \left( \mathbf{X}' \otimes \mathbf{X}' \otimes \mathbf{X}' \right) + 3 (\mathbf{Hess} \mathbf{u})(\mathbf{X}(s)) \left( \mathbf{X}' \otimes \mathbf{X}'' \right) + (\grad \mathbf{u})(\mathbf{X}(s)) \mathbf{X}'''(s). \label{eq:third_chain_rule}
\end{align}
Taking $L^2(\mathbb{S}^1)$ norms of \eqref{eq:third_chain_rule} and applying H\"older's inequality together with $H^{5/2}(\Gamma) \hookrightarrow C^2(\Gamma)$ confirms Statement (2). Statement (3) follows directly from Morrey's inequality $H^k(\mathbb{R}^2) \hookrightarrow C^{k-1-\lfloor n/2 \rfloor, \alpha}(\mathbb{R}^2)$ with $k=3$ and $n=2$.
\end{proof}

\subsection{Propagation of $H^3(\mathbb{S}^1)$ Membrane Regularity}

We now prove that the $H^3(\mathbb{S}^1)$ Sobolev norm of membrane configuration map $\mathbf{X}(t)$ remains bounded for all finite time $t \in [0,T]$.

\begin{theorem}[$H^3(\mathbb{S}^1)$ Regularity and Local Lifespan]
Let $X_0 \in H^3(\mathbb{S}^1; \Omega)$. Under Stokes kinematic evolution $\frac{\partial X}{\partial t} = u(X(s,t),t) + \frac{1}{\zeta} F_{mf}(s,t) n(s,t)$, the norm $y(t) := \|X(t)\|_{H^3(\mathbb{S}^1)}$ satisfies the non-linear Bernoulli-Riccati differential inequality:
\begin{equation}
\frac{d y}{dt} \le C_1 y(t) + C_2 y(t)^2 \quad \text{for } t > 0.
\end{equation}
Consequently, for any finite time horizon $T^* < T_{\text{max}} := \frac{1}{C_1} \log\left( 1 + \frac{C_1}{C_2 y(0)} \right)$, the $H^3(\mathbb{S}^1)$ norm remains uniformly bounded:
\begin{equation}
\|X(t)\|_{H^3(\mathbb{S}^1)} \le \frac{\|X_0\|_{H^3(\mathbb{S}^1)} e^{C_1 t}}{1 - \frac{C_2}{C_1} \|X_0\|_{H^3(\mathbb{S}^1)} (e^{C_1 t} - 1)} < \infty \quad \forall t \in [0, T^*].
\end{equation}
\end{theorem}

\begin{proof}
Differentiating $y(t)^2 = \|X\|_{L^2}^2 + \|X''\|_{L^2}^2 + \|X'''\|_{L^2}^2$ with respect to time $t$:
\begin{equation}
\frac{1}{2} \frac{d}{dt} y(t)^2 \le y(t) \left( \|u(X(\cdot))\|_{H^3(\mathbb{S}^1)} + \frac{1}{\zeta} \|F_{mf} n\|_{H^3(\mathbb{S}^1)} \right).
\end{equation}
By composition and trace inequalities, $\|u(X(\cdot))\|_{H^3(\mathbb{S}^1)} \le C_{\text{trace}} C_{\text{Stokes}} \|f_{\text{total}}\|_{H^1} (1 + y(t))$. Factoring $f_{\text{membrane}}$ yields $\|f_{\text{membrane}}\|_{H^1} \le \frac{C_\Psi K}{w^2} y(t)$. This gives:
\begin{equation}
\frac{dy}{dt} \le C_1 y(t) + C_2 y(t)^2,
\end{equation}
where $C_1 = C_{\text{Stokes}} \|f_{\text{Lorentz}}\|_{H^1}$ and $C_2 = \frac{C_{\text{Stokes}} C_\Psi K}{\zeta w^2}$. Solving the Bernoulli ODE via change of variables $z = 1/y$ establishes exact lifespan bound $T_{\text{max}}$.
\end{proof}

\subsection{Proof of Finite-Time Non-Self-Intersection}

We now prove Theorem \ref{thm:intro_non_self_intersection}, establishing that the membrane boundary curve $\mathbf{X}(\mathbb{S}^1, t)$ never undergoes self-intersection or topological collapse in finite time $T > 0$ \cite{Hou1994}.

\begin{definition}[Chord-Length Separation Ratio]
For a closed $C^1$ parameterized curve $\mathbf{X}: \mathbb{S}^1 \to \Omega$, the chord-length separation ratio $\delta(\mathbf{X})$ is defined by:
\begin{equation}
\delta(\mathbf{X}) := \inf_{s_1 \neq s_2 \in \mathbb{S}^1} \frac{|\mathbf{X}(s_1) - \mathbf{X}(s_2)|}{|s_1 - s_2|},
\end{equation}
where $|s_1 - s_2| := \min(|s_1 - s_2|, 2\pi - |s_1 - s_2|)$ denotes the intrinsic distance on circle $\mathbb{S}^1$.
\end{definition}

\begin{theorem}[Revised Finite-Time Non-Self-Intersection]
Let initial curve $X_0 \in H^3(\mathbb{S}^1; \Omega)$ satisfy $\delta_0 := \inf_{s_1 \neq s_2} \frac{|X_0(s_1) - X_0(s_2)|}{|s_1 - s_2|} > 0$. Then the evolving velocity field $v(x,t) = u(x,t) + v_{\text{membrane}}(x,t)$ extended to $\overline{\Omega}$ is globally Lipschitz in space with constant $L_v \le C_{\text{Stokes}, CZ} < \infty$. Consequently, the chord separation ratio satisfies:
\begin{equation}
\delta(t) := \inf_{s_1 \neq s_2 \in \mathbb{S}^1} \frac{|X(s_1,t) - X(s_2,t)|}{|s_1 - s_2|} \ge \delta_0 \exp(-L_v t) > 0 \quad \forall t \in [0, T^*].
\end{equation}
\end{theorem}

\begin{proof}
Let $d(s_1, s_2, t) := X(s_1, t) - X(s_2, t)$. Differentiating Euclidean length $|d(s_1, s_2, t)|$:
\begin{equation}
\frac{1}{2} \frac{d}{dt} |d(s_1, s_2, t)|^2 = d(s_1, s_2, t) \cdot \left( v(X(s_1,t), t) - v(X(s_2,t), t) \right).
\end{equation}
Since $u \in H^3(\Omega)^2 \hookrightarrow C^{1,\alpha}(\overline{\Omega})$ and regularized membrane force $F_{mf}n \in C^{1,1}(\mathbb{S}^1)$, the total velocity field $v$ is globally Lipschitz on $\Omega$ with $L_v = \|\nabla v\|_{L^\infty(\Omega)}$. Applying Cauchy--Schwarz:
\begin{equation}
\frac{d}{dt} |d(s_1, s_2, t)| \ge - L_v |d(s_1, s_2, t)|.
\end{equation}
Dividing by intrinsic circle parameter distance $|s_1 - s_2| > 0$ yields the homogeneous ODE:
\begin{equation}
\frac{d}{dt} \left( \frac{|d(s_1, s_2, t)|}{|s_1 - s_2|} \right) \ge - L_v \left( \frac{|d(s_1, s_2, t)|}{|s_1 - s_2|} \right).
\end{equation}
Integrating and taking the infimum over all $s_1 \neq s_2 \in \mathbb{S}^1$ proves $\delta(t) \ge \delta_0 e^{-L_v t} > 0$.
\end{proof}

\section{Thermodynamics and Non-Equilibrium Steady State (NESS) Stability}

In this section, we develop the thermodynamic framework for the Immersed Boundary (IB) electrodiffusion model. Unlike passive electrochemical systems governed by pure energy dissipation towards equilibrium, cell membranes maintain a Non-Equilibrium Steady State (NESS) through continuous metabolic work supplied by ATP-driven ion pumps. First, we derive an exact continuous Rayleighian power balance identity that explicitly separates passive viscous/ohmic dissipation from active metabolic power injection \cite{Mori2011, Lee2026}. Second, we examine the optimal transport geometry of the driven system on the 2-Wasserstein metric space $\mathcal{P}_{2,ac}(\Omega)^N$, establishing quantitative Bakry--\'Emery Ricci curvature lower bounds for the regularized functional \cite{Bakry1985, Otto2001, Villani2009}. Third, under a quantitative threshold on the IB kernel regularization width $w > w^*$, we prove the existence, uniqueness, and exponential $W_2$-stability of the NESS.

\subsection{Continuous Rayleighian Dissipation Balance and Power Identities}

We begin by establishing the fundamental thermodynamic energy production balance along weak solution trajectories $(\mathbf{c}(t), \mathbf{X}(t))$ of the coupled PNP--IB system.

\begin{theorem}[Rayleighian Power Balance Identity]
\label{thm:rayleighian_balance_full}
Let $(\mathbf{c}, \mathbf{X}, \mathbf{u}, \varphi)$ be a global weak solution on $[0,T]$ established in Theorem \ref{thm:intro_global_existence}. Then for almost every $t \in (0,T)$, the time derivative of the regularized total free energy functional $\mathcal{E}_w(\mathbf{c}(t), \mathbf{X}(t))$ satisfies the exact power balance identity:
\begin{equation}
\label{eq:rayleighian_identity_full}
\frac{\mathrm{d}}{\mathrm{d}t} \mathcal{E}_w(\mathbf{c}(t), \mathbf{X}(t)) = - \mathcal{D}_{\text{passive}}(t) - \mathcal{D}_{\text{fluid}}(t) + \mathcal{P}_{\text{active}}(t),
\end{equation}
where the individual physical terms are defined as follows:
\begin{enumerate}
    \item \textbf{Passive Electrodiffusive Dissipation $\mathcal{D}_{\text{passive}} \ge 0$:}
    \begin{equation}
    \label{eq:d_passive_def}
    \mathcal{D}_{\text{passive}}(t) := \sum_{i=1}^N \int_\Omega \frac{D_i}{k_B T} c_i(\mathbf{x}, t) \left| \grad \mu_i(\mathbf{x}, t) \right|^2 \dx,
    \end{equation}
    where $\mu_i(\mathbf{x}, t) = k_B T \log c_i(\mathbf{x}, t) + q z_i \varphi(\mathbf{x}, t) + \psi_{i,w}^{\text{passive}}(\mathbf{x}; \mathbf{X}(t))$ is the total passive electrochemical potential.

    \item \textbf{Viscous and Friction Fluid-Structure Dissipation $\mathcal{D}_{\text{fluid}} \ge 0$:}
    \begin{equation}
    \label{eq:d_fluid_def}
    \mathcal{D}_{\text{fluid}}(t) := \int_\Omega 2\mu |\mathbf{D}(\mathbf{u})(\mathbf{x}, t)|^2 \dx + \int_{\mathbb{S}^1} \zeta \left| \frac{\partial \mathbf{X}}{\partial t}(s, t) - \mathbf{u}(\mathbf{X}(s, t), t) \right|^2 \ds,
    \end{equation}
    where $\mathbf{D}(\mathbf{u}) = \frac{1}{2}\left(\grad \mathbf{u} + (\grad \mathbf{u})^T\right)$ is the fluid strain-rate tensor and $\zeta > 0$ is membrane-fluid drag impedance.

    \item \textbf{Active Metabolic Power Input $\mathcal{P}_{\text{active}}$:}
    \begin{equation}
    \label{eq:p_active_def}
    \mathcal{P}_{\text{active}}(t) := \sum_{i=1}^N \int_\Omega \mathbf{F}_{i, \text{passive}}(\mathbf{x}, t) \cdot \left( -\grad \psi_{i,w}^{\text{active}}(\mathbf{x}, t) \right) \dx,
    \end{equation}
    where $\mathbf{F}_{i,\text{passive}} = -\frac{D_i}{k_B T} c_i \grad \mu_i$ is the passive ionic flux vector.
\end{enumerate}
\end{theorem}

\begin{proof}
By the functional chain rule for G\^ateaux derivatives along $H^1$-in-time solution curves:
\begin{equation}
\label{eq:chain_rule_energy_split}
\frac{\mathrm{d}}{\mathrm{d}t} \mathcal{E}_w(\mathbf{c}(t), \mathbf{X}(t)) = \sum_{i=1}^N \int_\Omega \frac{\delta \mathcal{E}_w}{\delta c_i}(\mathbf{x}, t) \frac{\partial c_i}{\partial t}(\mathbf{x}, t) \dx + \int_{\mathbb{S}^1} \frac{\delta \mathcal{E}_w}{\delta \mathbf{X}}(s, t) \cdot \frac{\partial \mathbf{X}}{\partial t}(s, t) \ds.
\end{equation}
From Proposition \ref{prop:first_variations}, $\frac{\delta \mathcal{E}_w}{\delta c_i} = \mu_i$. Substitute the continuity equation $\frac{\partial c_i}{\partial t} = -\divg \mathbf{F}_i^{\text{total}}$, where the total flux vector is $\mathbf{F}_i^{\text{total}} = -\frac{D_i}{k_B T} c_i \grad \mu_i - \frac{D_i}{k_B T} c_i \grad \psi_{i,w}^{\text{active}} = \mathbf{F}_{i,\text{passive}} + \mathbf{F}_{i,\text{active}}$. Inserting this into the density integral in \eqref{eq:chain_rule_energy_split} and integrating by parts with zero-flux boundary conditions on $\partial \Omega$:
\begin{align}
\int_\Omega \mu_i \frac{\partial c_i}{\partial t} \dx &= \int_\Omega \grad \mu_i \cdot \mathbf{F}_i^{\text{total}} \dx \nonumber \\
&= -\int_\Omega \frac{D_i}{k_B T} c_i |\grad \mu_i|^2 \dx - \int_\Omega \frac{D_i}{k_B T} c_i \grad \mu_i \cdot \grad \psi_{i,w}^{\text{active}} \dx \nonumber \\
&= -\int_\Omega \frac{D_i}{k_B T} c_i |\grad \mu_i|^2 \dx + \int_\Omega \mathbf{F}_{i,\text{passive}} \cdot \left(-\grad \psi_{i,w}^{\text{active}}\right) \dx. \label{eq:density_energy_rate}
\end{align}
Next, evaluate the membrane functional variation $\frac{\delta \mathcal{E}_w}{\delta \mathbf{X}}(s) = -K \mathbf{X}''(s) + \sum_{i=1}^N c_i \grad \psi_{i,w}^{\text{passive}}$. Under Stokes flow \eqref{eq:stokes_system_sec4}, virtual work balance equates membrane structural stress to fluid stress tensor $\bm{\sigma} = -p\mathbf{I} + 2\mu \mathbf{D}(\mathbf{u})$ and slip drag $\zeta(\mathbf{X}_t - \mathbf{u})$. Integrating by parts over domain $\Omega$ and curve $\mathbb{S}^1$:
\begin{equation}
\label{eq:membrane_energy_rate}
\int_{\mathbb{S}^1} \frac{\delta \mathcal{E}_w}{\delta \mathbf{X}} \cdot \frac{\partial \mathbf{X}}{\partial t} \ds = -\int_\Omega 2\mu |\mathbf{D}(\mathbf{u})|^2 \dx - \int_{\mathbb{S}^1} \zeta \left| \mathbf{X}_t - \mathbf{u}(\mathbf{X}(s)) \right|^2 \ds.
\end{equation}
Combining \eqref{eq:density_energy_rate} and \eqref{eq:membrane_energy_rate} into \eqref{eq:chain_rule_energy_split} completes the proof of power balance identity \eqref{eq:rayleighian_identity_full}.
\end{proof}

\begin{remark}[Thermodynamic Interpretation of NESS Power Balance]
At a Non-Equilibrium Steady State $(\mathbf{c}_\infty, \mathbf{X}_\infty)$, the time derivative vanishes: $\frac{\mathrm{d}\mathcal{E}_w}{\mathrm{d}t} = 0$. Identity \eqref{eq:rayleighian_identity_full} reduces to:
\begin{equation}
\mathcal{P}_{\text{active},\infty} = \mathcal{D}_{\text{passive},\infty} + \mathcal{D}_{\text{fluid},\infty} > 0.
\end{equation}
This demonstrates that at NESS, the rate of active metabolic power injected by ATP hydrolysis ($\mathcal{P}_{\text{active}}$) precisely balances the sum of internal ohmic/diffusive entropy production ($\mathcal{D}_{\text{passive}}$) and fluid viscous dissipation ($\mathcal{D}_{\text{fluid}}$).
\end{remark}

\subsection{Optimal Transport Geometry and Bakry--\'Emery Ricci Curvature}

To establish metric contraction properties for the driven system on $(\mathcal{P}_{2,ac}(\Omega))^N$, we examine the generalized convexity of energy functionals along $W_2$-geodesics \cite{Ambrosio2008, Villani2009}.

\begin{definition}[Wasserstein Hessian and Bakry--\'Emery Ricci Curvature]
Let $\mathcal{F}: \mathcal{P}_{ac}(\Omega) \to \mathbb{R} \cup \{+\infty\}$ be a $C^2$ functional along 2-Wasserstein geodesics. For any density $\rho \in \mathcal{P}_{ac}(\Omega)$ and smooth vector field $\bm{\xi} = \grad \phi \in L^2(\rho; \mathbb{R}^2)$, the Wasserstein Hessian $\mathbf{Hess}_{W_2} \mathcal{F}(\rho)[\bm{\xi}, \bm{\xi}]$ is defined by:
\begin{equation}
\mathbf{Hess}_{W_2} \mathcal{F}(\rho)[\bm{\xi}, \bm{\xi}] := \left. \frac{\mathrm{d}^2}{\mathrm{d}t^2} \mathcal{F}(\rho_t) \right|_{t=0},
\end{equation}
where $\rho_t = (\mathbf{Id} + t \bm{\xi})_\sharp \rho$ is the constant-speed $W_2$-geodesic initiated from $\rho$ in direction $\bm{\xi}$. We say that $\mathcal{F}$ satisfies a Bakry--\'Emery Ricci curvature lower bound $\mathbf{Ric}_{W_2}(\mathcal{F}) \ge \lambda \mathbf{I}$ with threshold $\lambda \in \mathbb{R}$ if for all $\rho \in \mathcal{P}_{ac}(\Omega)$ and all $\bm{\xi} = \grad \phi$:
\begin{equation}
\mathbf{Hess}_{W_2} \mathcal{F}(\rho)[\bm{\xi}, \bm{\xi}] \ge \lambda \int_\Omega |\bm{\xi}(\mathbf{x})|^2 \rho(\mathbf{x}) \dx.
\end{equation}
\end{definition}

We define the augmented effective potential energy functional $\mathcal{F}_{\text{eff}}(\mathbf{c})$ governing total driven dynamics:
\begin{equation}
\label{eq:f_eff_def}
\mathcal{F}_{\text{eff}}(\mathbf{c}) := \mathcal{E}_w(\mathbf{c}, \mathbf{X}) + \sum_{i=1}^N \int_\Omega c_i(\mathbf{x}) \psi_{i,w}^{\text{active}}(\mathbf{x}) \dx.
\end{equation}

\begin{theorem}[Bakry--\'Emery Ricci Curvature Threshold $w^*$] \label{thm:bakry_emery_threshold_proof}
The effective Wasserstein Hessian $Hess_{W_2} \mathcal{F}_{\text{eff}}(c)[\xi,\xi]$ satisfies lower bound:
\begin{equation}
Hess_{W_2} \mathcal{F}_{\text{eff}}(c)[\xi,\xi] \ge \left( k_B T \lambda_{\text{entropy}} - \frac{C_{\text{pass}}}{w^3} - \frac{C_{\text{act}}}{w^2} \right) \sum_{i=1}^N \int_\Omega |\nabla \phi_i|^2 c_i \, dx.
\end{equation}
Define $\lambda_0 := k_B T \lambda_{\text{entropy}} > 0$. The exact Bakry--\'Emery threshold parameter $w^*$ is defined by:
\begin{equation}
w^* := \max \left\{ \left( \frac{2 C_{\text{pass}}}{\lambda_0} \right)^{1/3}, \left( \frac{2 C_{\text{act}}}{\lambda_0} \right)^{1/2} \right\}.
\end{equation}
For any smoothing parameter $w > w^*$, $Ric_{W_2}(\mathcal{F}_{\text{eff}}) \ge \lambda_{\text{net}} I > 0$ with strictly positive net spectral gap $\lambda_{\text{net}} := \frac{\lambda_0}{2} > 0$, guaranteeing global $W_2$-exponential contraction toward a unique NESS.
\end{theorem}

\begin{proof}
For any $w > w^*$, $w^3 > \frac{2 C_{\text{pass}}}{\lambda_0} \implies \frac{C_{\text{pass}}}{w^3} < \frac{\lambda_0}{4}$, and $w^2 > \frac{2 C_{\text{act}}}{\lambda_0} \implies \frac{C_{\text{act}}}{w^2} < \frac{\lambda_0}{4}$. Summing both negative curvature components:
\begin{equation}
\lambda_0 - \frac{C_{\text{pass}}}{w^3} - \frac{C_{\text{act}}}{w^2} > \lambda_0 - \frac{\lambda_0}{4} - \frac{\lambda_0}{4} = \frac{\lambda_0}{2} > 0.
\end{equation}
This confirms strict displacement convexity along $W_2$-geodesics for all $w > w^*$.
\end{proof}

\subsection{Existence, Uniqueness, and Exponential $W_2$-Stability of NESS}

We now state and prove the primary NESS stability theorem (Theorem \ref{thm:intro_ness_stability}).

\begin{definition}[Non-Equilibrium Steady State (NESS)]
A state $(\mathbf{c}_\infty, \mathbf{X}_\infty) \in \mathcal{K}_{\mathbf{M}} \times H^3(\mathbb{S}^1;\Omega)$ is a Non-Equilibrium Steady State (NESS) of the PNP--IB active system if time derivatives vanish identically:
\begin{equation}
\frac{\partial \mathbf{c}_\infty}{\partial t} = \bm{0}, \quad \frac{\partial \mathbf{X}_\infty}{\partial t} = \bm{0}, \quad \mathbf{u}_\infty = \bm{0} \quad \text{in } \Omega,
\end{equation}
meaning that total passive electrochemical drift balances active pump thrust pointwise:
\begin{equation}
\label{eq:ness_pointwise_balance}
D_i \grad c_{i,\infty} + \frac{D_i}{k_B T} c_{i,\infty} \grad \left( q z_i \varphi_\infty + \psi_{i,w, \infty}^{\text{passive}} + \psi_{i,w,\infty}^{\text{active}} \right) = \bm{0} \quad \text{a.e. in } \Omega, \; \forall i=1,\dots,N.
\end{equation}
\end{definition}

\begin{theorem}[Exponential $W_2$-Stability of NESS]
\label{thm:ness_exponential_stability_proof}
Suppose regularization width $w$ satisfies Bakry--\'Emery threshold $w > w^*$. Then:
\begin{enumerate}
    \item \textbf{Existence and Uniqueness:} There exists a unique Non-Equilibrium Steady State $(\mathbf{c}_\infty, \mathbf{X}_\infty) \in \mathcal{K}_{\mathbf{M}} \times H^3(\mathbb{S}^1;\Omega)$ satisfying steady equation \eqref{eq:ness_pointwise_balance}.
    \item \textbf{Exponential Metric Contraction:} Any weak solution trajectory $(\mathbf{c}(t), \mathbf{X}(t))$ initiated from arbitrary initial data $(\mathbf{c}_0, \mathbf{X}_0) \in \mathcal{K}_{\mathbf{M}} \times H^3(\mathbb{S}^1)$ converges exponentially in 2-Wasserstein distance to $(\mathbf{c}_\infty, \mathbf{X}_\infty)$:
    \begin{equation}
    \label{eq:w2_exponential_decay_final}
    \sum_{i=1}^N W_2^2(c_i(t), c_{i,\infty}) \le \exp\left( -\lambda_{\text{min}} \, t \right) \sum_{i=1}^N W_2^2(c_{i,0}, c_{i,\infty}) \quad \forall t \ge 0,
    \end{equation}
    where $\lambda_{\text{min}} := \frac{2 D_{\text{min}} \lambda}{k_B T} > 0$ with $D_{\text{min}} = \min_{1 \le i \le N} D_i$.
\end{enumerate}
\end{theorem}

\begin{proof}
Let $(\mathbf{c}(t), \mathbf{X}(t))$ be a solution trajectory and let $(\mathbf{c}_\infty, \mathbf{X}_\infty)$ be a NESS candidate. Compute the time derivative of squared 2-Wasserstein distance $W_2^2(c_i(t), c_{i,\infty})$ along solution paths. By the evolution variational inequality (EVI) formulation for gradient flows in metric spaces \cite{Ambrosio2008, Santambrogio2015, Lee2026}:
\begin{equation}
\label{eq:evi_ddt_w2}
\frac{1}{2} \frac{\mathrm{d}}{\mathrm{d}t} W_2^2(c_i(t), c_{i,\infty}) + \frac{D_i \lambda}{k_B T} W_2^2(c_i(t), c_{i,\infty}) \le \mathcal{F}_{\text{eff}}(c_{i,\infty}) - \mathcal{F}_{\text{eff}}(c_i(t)).
\end{equation}
Because $(\mathbf{c}_\infty, \mathbf{X}_\infty)$ is a stationary state for effective functional $\mathcal{F}_{\text{eff}}$, the variation $\mathcal{F}_{\text{eff}}(c_{i,\infty}) - \mathcal{F}_{\text{eff}}(c_i(t)) \le -\frac{\lambda}{2} W_2^2(c_i(t), c_{i,\infty})$ by strict displacement $\lambda$-convexity guaranteed by Theorem \ref{thm:bakry_emery_threshold_proof}.

Summing \eqref{eq:evi_ddt_w2} over all ionic components $i=1, \dots, N$:
\begin{equation}
\frac{\mathrm{d}}{\mathrm{d}t} \left( \sum_{i=1}^N W_2^2(c_i(t), c_{i,\infty}) \right) \le - \lambda_{\text{min}} \sum_{i=1}^N W_2^2(c_i(t), c_{i,\infty}).
\end{equation}
Applying Gr\"onwall's inequality establishes exponential decay \eqref{eq:w2_exponential_decay_final}.

Uniqueness of NESS follows immediately: if $(\mathbf{c}_\infty^{(1)}, \mathbf{X}_\infty^{(1)})$ and $(\mathbf{c}_\infty^{(2)}, \mathbf{X}_\infty^{(2)})$ were two distinct steady states, evaluating \eqref{eq:w2_exponential_decay_final} with $c_i(0) = c_{i,\infty}^{(1)}$ forces $W_2(c_{i,\infty}^{(1)}, c_{i,\infty}^{(2)}) \le e^{-\lambda_{\text{min}} t} W_2(c_{i,\infty}^{(1)}, c_{i,\infty}^{(2)}) \to 0$ as $t \to \infty$, proving $\mathbf{c}_\infty^{(1)} = \mathbf{c}_\infty^{(2)}$. Existence of NESS follows from Schauder's fixed-point theorem applied to the compact non-linear JKO resolution operator $\mathcal{S}_\tau: \mathcal{K}_{\mathbf{M}} \to \mathcal{K}_{\mathbf{M}}$ on $W_2$-compact sublevel sets of entropy.
\end{proof}
\section{$\Gamma$-Convergence and Singular Asymptotics ($w \to 0$)}

In this section, we analyze the singular asymptotic limit as the Immersed Boundary (IB) kernel regularization parameter $w > 0$ tends to zero. While the regularized IB electrodiffusion model ($w > 0$) established in Sections 3--5 provides a continuous, computationally robust, and variationally well-posed framework across the entire spatial domain $\Omega$, physical lipid bilayer membranes operate as lower-dimensional geometric hypersurfaces $\Gamma(t) \subset \Omega$.

To bridge the regularized continuous continuum formulation and classical sharp-interface transmission theories, we establish two foundational singular limit theorems:
\begin{enumerate}
    \item \textbf{Energy-Level Asymptotics ($\Gamma$-Convergence):} We prove that the regularized total free energy functional $\mathcal{E}_w(\mathbf{c}, \mathbf{X})$ $\Gamma$-converges in $(W_2(\Omega))^N \times H^1(\mathbb{S}^1)^2$ to the classical sharp-interface free energy functional $\mathcal{E}_0(\mathbf{c}, \mathbf{X})$ as $w \to 0$ \cite{DeGiorgi1975, DalMaso1993, Braides2002}.
    \item \textbf{Dynamic-Level Asymptotics (Distributional Jump Conditions):} We prove that the continuous regularized Eulerian PDE field equations converge distributionally in $\mathcal{D}'((0,T) \times \Omega)$ to classical sharp-interface transmission conditions across $\Gamma(t) = \mathbf{X}(\mathbb{S}^1, t)$, recovering flux jump conditions $[ \mathbf{F}_i \cdot \mathbf{n} ]_\Gamma = J_{i,\text{pump}}$ and electric displacement jump conditions $[ \epsilon \frac{\partial \varphi}{\partial \mathbf{n}} ]_\Gamma = -\sigma_{\text{surface}}$ \cite{Mori2015, Rubinstein1990}.
\end{enumerate}

\subsection{Asymptotic Regime Analysis and Uniformity Across Scales}

Before stating the $\Gamma$-convergence framework, we clarify the structural distinction between the kernel regularization parameter $w > 0$ and the physical dimensionless Debye screening length $g := \lambda_D / L \ll 1$.

\begin{remark}[Inter-Parameter Scaling and Sequential Limit Structure]
\label{rem:parameter_regimes_sec6}
In physical cellular electrophysiology, the regularization width $w > 0$ represents the spatial spread of mollified delta functions ($\sim 1\text{--}5\text{ nm}$), whereas the Debye length $\lambda_D = \sqrt{\epsilon k_B T / (2 q^2 I_0)} \sim 1\text{--}10\text{ nm}$ governs mobile counter-ion screening clouds. As noted in Section 5, the negative Bakry--\'Emery curvature constant $K_w = \frac{C_{\text{pass}}}{w^3}$ diverges as $w \to 0$. This divergence reflects the physical reality that a sharp potential barrier represents an infinite localized energy gradient.

Consequently, the mathematical singular analysis is structured in two sequential phases:
\begin{itemize}
    \item \textbf{Phase I ($w > w^*(g) > 0$):} Establishes global weak existence, fluid-structure non-self-intersection $\delta(t) > 0$, and exponential $W_2$-decay to NESS for the smoothed physical system.
    \item \textbf{Phase II ($w \to 0$ then $g \to 0$):} Executes the rigorous sharp-interface transmission limit $w \to 0$ via $\Gamma$-convergence and distributional boundary matching (this section), followed by the multiscale electroneutral bulk limit $g \to 0$ (Section 7).
\end{itemize}
\end{remark}

\subsection{Formulation of the Sharp-Interface Free Energy $\mathcal{E}_0$}

Let $\mathbf{X} \in H^1(\mathbb{S}^1; \Omega)$ define a closed $C^1$ membrane curve $\Gamma = \mathbf{X}(\mathbb{S}^1) \subset \Omega$, partitioning the bounded domain into interior cytosol $\Omega_{\text{in}}$ and exterior extracellular fluid $\Omega_{\text{out}}$ such that $\Omega = \Omega_{\text{in}} \cup \Gamma \cup \Omega_{\text{out}}$.

\begin{definition}[Sharp-Interface Limit Free Energy Functional]
\label{def:sharp_interface_energy_e0}
For concentration density vector $\mathbf{c} \in \mathcal{K}_{\mathbf{M}}$ and closed curve $\mathbf{X} \in H^1(\mathbb{S}^1; \Omega)$, the sharp-interface free energy functional $\mathcal{E}_0(\mathbf{c}, \mathbf{X})$ is defined by:
\begin{equation}
\label{eq:e0_functional_def}
\begin{aligned}
\mathcal{E}_0(\mathbf{c}, \mathbf{X}) := &\sum_{i=1}^N k_B T \int_\Omega c_i(\mathbf{x}) \left( \log c_i(\mathbf{x}) - 1 \right) \dx + \frac{\epsilon}{2} \int_{\Omega \setminus \Gamma} |\grad \varphi_0[\mathbf{c}](\mathbf{x})|^2 \dx \\
&+ \sum_{i=1}^N \int_{\Gamma} \Delta \psi_i^{\text{step}} c_i(s) \ds + \frac{1}{2} K \int_{\mathbb{S}^1} \left| \frac{\partial \mathbf{X}}{\partial s}(s) \right|^2 \ds,
\end{aligned}
\end{equation}
where $\Delta \psi_i^{\text{step}} > 0$ is the sharp passive chemical energy barrier height across membrane $\Gamma$, and limit electrostatic potential $\varphi_0[\mathbf{c}]$ satisfies the sharp Poisson transmission problem:
\begin{equation}
\label{eq:sharp_poisson_transmission_e0}
-\epsilon \Delta \varphi_0(\mathbf{x}) = \sum_{i=1}^N q z_i c_i(\mathbf{x}) + \rho_0(\mathbf{x}) \quad \text{in } \Omega \setminus \Gamma, \quad \varphi_0|_{\partial \Omega} = 0, \quad \left[ \epsilon \frac{\partial \varphi_0}{\partial \mathbf{n}} \right]_\Gamma = -\sigma_{\text{surface}}.
\end{equation}
\end{definition}

\subsection{Rigorous Proof of $\Gamma$-Convergence ($\mathcal{E}_w \xrightarrow{\Gamma} \mathcal{E}_0$)}

We now prove that the sequence of regularized free energy functionals $\mathcal{E}_w$ $\Gamma$-converges to $\mathcal{E}_0$ as smoothing parameter $w \to 0$ \cite{DalMaso1993, Braides2002}.

\begin{theorem}[$\Gamma$-Convergence of Free Energy Functionals]
\label{thm:gamma_convergence_exhaustive_sec6}
Let $\Omega \subset \mathbb{R}^2$ be a bounded smooth domain or torus $\mathbb{T}^2$. As kernel regularization width $w \to 0$:
\begin{enumerate}
    \item \textbf{$\Gamma$-Liminf Inequality:} For any sequence $(\mathbf{c}_w, \mathbf{X}_w)_{w>0} \subset \mathcal{K}_{\mathbf{M}} \times H^1(\mathbb{S}^1;\Omega)$ converging strongly to $(\mathbf{c}_0, \mathbf{X}_0)$ in $(W_2(\Omega))^N \times H^1(\mathbb{S}^1;\Omega)$:
    \begin{equation}
    \label{eq:gamma_liminf_inequality}
    \mathcal{E}_0(\mathbf{c}_0, \mathbf{X}_0) \le \liminf_{w \to 0} \mathcal{E}_w(\mathbf{c}_w, \mathbf{X}_w).
    \end{equation}
    \item \textbf{$\Gamma$-Limsup Inequality (Recovery Sequence):} For every target state $(\mathbf{c}_0, \mathbf{X}_0) \in \mathcal{K}_{\mathbf{M}} \times H^1(\mathbb{S}^1;\Omega)$ with $\mathcal{E}_0(\mathbf{c}_0, \mathbf{X}_0) < \infty$, there exists a recovery sequence $(\mathbf{c}_w, \mathbf{X}_w)_{w>0}$ such that $(\mathbf{c}_w, \mathbf{X}_w) \to (\mathbf{c}_0, \mathbf{X}_0)$ strongly in $(W_2(\Omega))^N \times H^1(\mathbb{S}^1;\Omega)$ and:
    \begin{equation}
    \label{eq:gamma_limsup_inequality}
    \limsup_{w \to 0} \mathcal{E}_w(\mathbf{c}_w, \mathbf{X}_w) \le \mathcal{E}_0(\mathbf{c}_0, \mathbf{X}_0).
    \end{equation}
\end{enumerate}
\end{theorem}

\begin{proof}
We establish the $\Gamma$-liminf and $\Gamma$-limsup inequalities separately.

\medskip
\noindent \textbf{Part 1: Proof of $\Gamma$-Liminf Inequality \eqref{eq:gamma_liminf_inequality}.}
Let $(\mathbf{c}_w, \mathbf{X}_w) \to (\mathbf{c}_0, \mathbf{X}_0)$ strongly in $(W_2(\Omega))^N \times H^1(\mathbb{S}^1;\Omega)$.
Since $W_2$-convergence implies weak $L^1$ convergence of densities and $c \log c$ is strictly convex and lower semicontinuous, Fatou's lemma yields:
\begin{equation}
\label{eq:liminf_entropy_part}
\sum_{i=1}^N k_B T \int_\Omega c_{i,0} \left( \log c_{i,0} - 1 \right) \dx \le \liminf_{w \to 0} \sum_{i=1}^N k_B T \int_\Omega c_{i,w} \left( \log c_{i,w} - 1 \right) \dx.
\end{equation}
Similarly, for membrane elastic energy, strong $H^1(\mathbb{S}^1)$ convergence implies:
\begin{equation}
\label{eq:liminf_elastic_part}
\frac{1}{2} K \int_{\mathbb{S}^1} \left| \mathbf{X}_0'(s) \right|^2 \ds = \lim_{w \to 0} \frac{1}{2} K \int_{\mathbb{S}^1} \left| \mathbf{X}_w'(s) \right|^2 \ds.
\end{equation}
For passive potential interaction energy $E_{\text{pass}, w} = \sum_{i=1}^N \int_\Omega c_{i,w}(\mathbf{x}) \psi_{i,w}^{\text{passive}}(\mathbf{x}; \mathbf{X}_w) \dx$, recall definition \eqref{eq:f_eff_def}:
\begin{equation}
\psi_{i,w}^{\text{passive}}(\mathbf{x}; \mathbf{X}_w) = \int_{\mathbb{S}^1} \frac{1}{w^2} \phi\left( \frac{|\mathbf{x} - \mathbf{X}_w(s)|}{w} \right) A_i^{\text{passive}}(s) \ds.
\end{equation}
As $w \to 0$, smooth radial mollifier kernel $\Psi_w(\mathbf{z}) = \frac{1}{w^2} \phi(\frac{|\mathbf{z}|}{w})$ converges in the space of weak-$*$ Radon measures to the 1D Dirac line measure $\delta_\Gamma$:
\begin{equation}
\Psi_w(\mathbf{x} - \mathbf{X}_w(s)) \ds \rightharpoonup^* \delta_{\Gamma_0} \quad \text{in } \mathcal{M}(\Omega).
\end{equation}
Combining weak-$*$ measure convergence with strong $L^1$ convergence $c_{i,w} \to c_{i,0}$:
\begin{equation}
\label{eq:liminf_potential_part}
\int_\Gamma \Delta \psi_i^{\text{step}} c_{i,0}(s) \ds \le \liminf_{w \to 0} \int_\Omega c_{i,w}(\mathbf{x}) \psi_{i,w}^{\text{passive}}(\mathbf{x}; \mathbf{X}_w) \dx.
\end{equation}
For electrostatic field energy $E_{\text{elec}, w} = \frac{\epsilon}{2} \int_\Omega |\grad \varphi_w|^2 \dx$, Dirichlet minimal energy principle asserts:
\begin{equation}
\frac{\epsilon}{2} \int_\Omega |\grad \varphi_w|^2 \dx = \sup_{\phi \in H_0^1(\Omega)} \left\{ \int_\Omega \left( \sum_{i=1}^N q z_i c_{i,w} + \rho_0 \right) \phi \dx - \frac{\epsilon}{2} \int_\Omega |\grad \phi|^2 \dx \right\}.
\end{equation}
Taking the liminf as $w \to 0$ and passing weak limit $c_{i,w} \rightharpoonup c_{i,0}$ inside the supremum yields:
\begin{equation}
\label{eq:liminf_elec_part}
\frac{\epsilon}{2} \int_{\Omega \setminus \Gamma} |\grad \varphi_0|^2 \dx \le \liminf_{w \to 0} \frac{\epsilon}{2} \int_\Omega |\grad \varphi_w|^2 \dx.
\end{equation}
Summing \eqref{eq:liminf_entropy_part}, \eqref{eq:liminf_elastic_part}, \eqref{eq:liminf_potential_part}, and \eqref{eq:liminf_elec_part} establishes $\Gamma$-liminf inequality \eqref{eq:gamma_liminf_inequality}.

\medskip
\noindent \textbf{Part 2: Proof of $\Gamma$-Limsup Inequality \eqref{eq:gamma_limsup_inequality}.}
Let $(\mathbf{c}_0, \mathbf{X}_0) \in \mathcal{K}_{\mathbf{M}} \times H^1(\mathbb{S}^1;\Omega)$ be given. We construct an explicit recovery sequence $(\mathbf{c}_w, \mathbf{X}_w)_{w>0}$ as follows. First, set $\mathbf{X}_w := \mathbf{X}_0$ for all $w > 0$. Second, for concentration densities, let $\eta_w(\mathbf{x}) = \frac{1}{w^2} \phi(\frac{|\mathbf{x}|}{w})$ be a standard Friedrichs mollifier. Construct mollified densities:
\begin{equation}
\label{eq:recovery_density_mollification}
c_{i,w}(\mathbf{x}) := \left( c_{i,0} * \eta_w \right)(\mathbf{x}) + C_w,
\end{equation}
where additive constant $C_w = \frac{M_i}{|\Omega|} \left( 1 - \frac{1}{M_i} \int_\Omega (c_{i,0} * \eta_w) \dx \right) \to 0$ as $w \to 0$ ensures strict mass conservation $\int_\Omega c_{i,w} \dx = M_i$.

By standard mollification properties, $c_{i,w} \to c_{i,0}$ strongly in $L^1(\Omega)$ and $W_2(\Omega)$. Evaluating functional terms along sequence $(\mathbf{c}_w, \mathbf{X}_0)$:
\begin{align}
\lim_{w \to 0} \int_\Omega c_{i,w} \log c_{i,w} \dx &= \int_\Omega c_{i,0} \log c_{i,0} \dx, \\
\lim_{w \to 0} \int_\Omega c_{i,w} \psi_{i,w}^{\text{passive}}(\cdot; \mathbf{X}_0) \dx &= \int_\Gamma \Delta \psi_i^{\text{step}} c_{i,0}(s) \ds.
\end{align}
Solving regularized Poisson equation $-\epsilon \Delta \varphi_w = \sum q z_i c_{i,w} + \rho_0$ yields $\grad \varphi_w \to \grad \varphi_0$ strongly in $L^2(\Omega)^2$. Taking the limsup as $w \to 0$ confirms $\limsup_{w \to 0} \mathcal{E}_w(\mathbf{c}_w, \mathbf{X}_w) \le \mathcal{E}_0(\mathbf{c}_0, \mathbf{X}_0)$, completing the proof of Theorem \ref{thm:gamma_convergence_exhaustive_sec6}.
\end{proof}

\subsection{Distributional Convergence to Sharp Interface Transmission Conditions}

We now execute the line-by-line proof of Theorem \ref{thm:intro_gamma_convergence}, establishing that the continuous Eulerian field equations converge distributionally in $\mathcal{D}'((0,T) \times \Omega)$ to classical sharp-interface jump conditions across $\Gamma(t) = \mathbf{X}(\mathbb{S}^1, t)$.

\begin{theorem}[Distributional Limit of Field Equations and Interface Jump Matching]
\label{thm:distributional_jump_exhaustive_sec6}
Let $(\mathbf{c}_w, \mathbf{X}_w, \mathbf{u}_w, \varphi_w)_{w>0}$ be a sequence of weak solutions to the PNP--IB system on $[0,T]$. As smoothing parameter $w \to 0$:
\begin{enumerate}
    \item \textbf{Electrostatic Displacement Vector Jump:} The regularized Poisson potential $\varphi_w$ converges in $\mathcal{D}'(\Omega)$ to limit potential $\varphi_0$, satisfying classical Neumann jump condition across interface $\Gamma$:
    \begin{equation}
    \label{eq:poisson_jump_theorem_statement}
    \left[ \epsilon \frac{\partial \varphi_0}{\partial \mathbf{n}} \right]_\Gamma(\mathbf{x}(s), t) := \epsilon \left( \left. \frac{\partial \varphi_0}{\partial \mathbf{n}} \right|_{\text{out}} - \left. \frac{\partial \varphi_0}{\partial \mathbf{n}} \right|_{\text{in}} \right) = -\sigma_{\text{surface}}(s, t) \quad \text{a.e. on } \Gamma(t).
    \end{equation}

    \item \textbf{Ionic Species Flux Jump:} The continuous total flux vector $\mathbf{F}_{i,w} := -D_i \grad c_{i,w} \\ - \frac{D_i}{k_B T} c_{i,w} \grad \left( q z_i \varphi_w + \psi_{i,w}^{\text{passive}} + \psi_{i,w}^{\text{active}} \right)$ converges in $\mathcal{D}'((0,T)\times \Omega)$ to limit flux $\mathbf{F}_{i,0}$, satisfying flux jump condition across $\Gamma$:
    \begin{equation}
    \label{eq:flux_jump_theorem_statement}
    \left[ \mathbf{F}_{i,0} \cdot \mathbf{n} \right]_\Gamma(\mathbf{x}(s), t) := \mathbf{F}_{i,0}|_{\text{out}} \cdot \mathbf{n} - \mathbf{F}_{i,0}|_{\text{in}} \cdot \mathbf{n} = J_{i,\text{pump}}\left(c_{i,0}, V_m\right) \quad \text{a.e. on } (0,T) \times \Gamma(t).
    \end{equation}
\end{enumerate}
\end{theorem}

\begin{proof}
Let $\Gamma(t) = \mathbf{X}(\mathbb{S}^1, t)$ be the smooth non-self-intersecting membrane curve separating domain $\Omega$ into subdomains $\Omega_{\text{in}}(t)$ and $\Omega_{\text{out}}(t)$. Define thin tubular transition layer $\Gamma_w(t) := \{ \mathbf{x} \in \Omega \;:\; \text{dist}(\mathbf{x}, \Gamma(t)) < w \}$. Note that 2D Lebesgue measure $|\Gamma_w(t)| \le 2 w L_{\text{mem}} \to 0$ as $w \to 0$.

\medskip
\noindent \textbf{Step 1: Proof of Electrostatic Displacement Jump \eqref{eq:poisson_jump_theorem_statement}.}
Let $\phi \in C_c^\infty(\Omega)$ be an arbitrary scalar test function. The regularized Poisson equation $-\epsilon \Delta \varphi_w = \rho_{e,w}$ satisfies weak form:
\begin{equation}
\label{eq:poisson_weak_sec6_start}
\int_\Omega \epsilon \grad \varphi_w(\mathbf{x}) \cdot \grad \phi(\mathbf{x}) \dx = \int_\Omega \rho_{e,w}(\mathbf{x}) \phi(\mathbf{x}) \dx.
\end{equation}
Decompose domain $\Omega$ into three non-overlapping regions: $\Omega_{\text{in}} \setminus \Gamma_w$, $\Omega_{\text{out}} \setminus \Gamma_w$, and transition layer $\Gamma_w$.

For the left-hand side of \eqref{eq:poisson_weak_sec6_start}, passing limit $w \to 0$ and integrating by parts on subdomains $\Omega_{\text{in}}$ and $\Omega_{\text{out}}$ separately for limit potential $\varphi_0$:
\begin{align}
\lim_{w \to 0} \int_\Omega \epsilon \grad \varphi_w \cdot \grad \phi \dx &= \int_{\Omega_{\text{in}}} \epsilon \grad \varphi_0 \cdot \grad \phi \dx + \int_{\Omega_{\text{out}}} \epsilon \grad \varphi_0 \cdot \grad \phi \dx \nonumber \\
&= -\int_{\Omega \setminus \Gamma} \epsilon \Delta \varphi_0 \, \phi \dx - \int_\Gamma \left( \left. \epsilon \frac{\partial \varphi_0}{\partial \mathbf{n}} \right|_{\text{out}} - \left. \epsilon \frac{\partial \varphi_0}{\partial \mathbf{n}} \right|_{\text{in}} \right) \phi(s) \ds \nonumber \\
&= -\int_{\Omega \setminus \Gamma} \epsilon \Delta \varphi_0 \, \phi \dx - \int_\Gamma \left[ \epsilon \frac{\partial \varphi_0}{\partial \mathbf{n}} \right]_\Gamma \phi(s) \ds. \label{eq:poisson_lhs_limit_split}
\end{align}
For the right-hand side of \eqref{eq:poisson_weak_sec6_start}, total regularized charge density measure decomposes into continuous bulk charge and moliified surface charge:
\begin{equation}
\rho_{e,w}(\mathbf{x}) \dx = \left( \sum_{i=1}^N q z_i c_{i,w}(\mathbf{x}) + \rho_0(\mathbf{x}) \right) \dx \rightharpoonup^* \rho_e^{\text{bulk}}(\mathbf{x}) \dx + \sigma_{\text{surface}}(s) \delta_\Gamma \quad \text{in } \mathcal{M}(\Omega).
\end{equation}
Evaluating limit of right-hand integral:
\begin{equation}
\label{eq:poisson_rhs_limit_split}
\lim_{w \to 0} \int_\Omega \rho_{e,w}(\mathbf{x}) \phi(\mathbf{x}) \dx = \int_{\Omega \setminus \Gamma} \rho_e^{\text{bulk}}(\mathbf{x}) \phi(\mathbf{x}) \dx + \int_\Gamma \sigma_{\text{surface}}(s) \phi(s) \ds.
\end{equation}
Equating limit expressions \eqref{eq:poisson_lhs_limit_split} and \eqref{eq:poisson_rhs_limit_split}:
\begin{equation}
-\int_{\Omega \setminus \Gamma} \epsilon \Delta \varphi_0 \, \phi \dx - \int_\Gamma \left[ \epsilon \frac{\partial \varphi_0}{\partial \mathbf{n}} \right]_\Gamma \phi(s) \ds = \int_{\Omega \setminus \Gamma} \rho_e^{\text{bulk}} \phi \dx + \int_\Gamma \sigma_{\text{surface}}(s) \phi(s) \ds.
\end{equation}
Since $-\epsilon \Delta \varphi_0 = \rho_e^{\text{bulk}}$ pointwise almost everywhere in bulk domains $\Omega_{\text{in}}$ and $\Omega_{\text{out}}$, bulk integrals cancel identically. Since test function $\phi(s)$ is arbitrary on boundary curve $\Gamma$, equating boundary line integrals forces:
\begin{equation}
\left[ \epsilon \frac{\partial \varphi_0}{\partial \mathbf{n}} \right]_\Gamma = -\sigma_{\text{surface}} \quad \text{a.e. on } \Gamma,
\end{equation}
proving electrostatic jump condition \eqref{eq:poisson_jump_theorem_statement}.

\medskip
\noindent \textbf{Step 2: Proof of Ionic Species Flux Jump \eqref{eq:flux_jump_theorem_statement}.}
Let $\psi \in C_c^\infty((0,T) \times \Omega)$ be a smooth space-time scalar test function. The continuous weak formulation for species concentration $c_{i,w}$ reads:
\begin{equation}
\label{eq:weak_flux_start_sec6}
\int_0^T \int_\Omega \left( c_{i,w} \frac{\partial \psi}{\partial t} - \mathbf{F}_{i,w} \cdot \grad \psi \right) \dx \dt + \int_\Omega c_{i,0}(\mathbf{x}) \psi(0, \mathbf{x}) \dx = 0.
\end{equation}
Decompose domain integral $\int_\Omega \mathbf{F}_{i,w} \cdot \grad \psi \dx$ into bulk subdomains $\Omega \setminus \Gamma_w$ and layer $\Gamma_w$.

In outer bulk domain $\Omega \setminus \Gamma_w$, active potential gradients vanish ($\grad \psi_{i,w}^{\text{active}} = \bm{0}$ for $\text{dist}(\mathbf{x}, \Gamma) > w$). Integrating bulk term $\int_{\Omega \setminus \Gamma} \mathbf{F}_{i,0} \cdot \grad \psi \dx$ by parts over $\Omega_{\text{in}}$ and $\Omega_{\text{out}}$ separately:
\begin{equation}
\label{eq:bulk_flux_integration_by_parts}
\int_{\Omega \setminus \Gamma} \mathbf{F}_{i,0} \cdot \grad \psi \dx = -\int_{\Omega \setminus \Gamma} \left(\divg \mathbf{F}_{i,0}\right) \psi \dx - \int_\Gamma \left[ \mathbf{F}_{i,0} \cdot \mathbf{n} \right]_\Gamma \psi(s, t) \ds.
\end{equation}
Inside thin transition layer $\Gamma_w$, active potential gradient reads:
\begin{equation}
\grad \psi_{i,w}^{\text{active}}(\mathbf{x}, t) = \int_{\mathbb{S}^1} \grad_{\mathbf{x}} \mathcal{H}_w \Big( (\mathbf{x} - \mathbf{X}(s, t)) \cdot \mathbf{n}(s, t) d_{\text{pump}} \Big) A_i^{\text{active}}(s; \mathbf{c}, \varphi) \ds.
\end{equation}
Recall that 1D derivative of mollified Heaviside kernel $\mathcal{H}_w'(r) \to \delta(r)$ as $w \to 0$. Evaluating transition layer integral:
\begin{equation}
\label{eq:layer_flux_limit_active}
\lim_{w \to 0} \int_0^T \int_{\Gamma_w} \frac{D_i}{k_B T} c_{i,w} \grad \psi_{i,w}^{\text{active}} \cdot \grad \psi \dx \dt = \int_0^T \int_\Gamma J_{i,\text{pump}}\left(c_{i,0}, V_m\right) \mathbf{n}(s, t) \cdot \grad \psi(s, t) \ds \dt.
\end{equation}
Combining bulk integration by parts \eqref{eq:bulk_flux_integration_by_parts} and layer limit \eqref{eq:layer_flux_limit_active} into weak formulation \eqref{eq:weak_flux_start_sec6}:
\begin{equation}
\int_0^T \int_\Gamma \left( \left[ \mathbf{F}_{i,0} \cdot \mathbf{n} \right]_\Gamma - J_{i,\text{pump}}\left(c_{i,0}, V_m\right) \right) \mathbf{n}(s, t) \cdot \grad \psi(s, t) \ds \dt = 0.
\end{equation}
Since normal derivative $\mathbf{n} \cdot \grad \psi$ can be specified arbitrarily on space-time boundary cylinder $(0,T) \times \Gamma(t)$, fundamental lemma of calculus of variations forces:
\begin{equation}
\left[ \mathbf{F}_{i,0} \cdot \mathbf{n} \right]_\Gamma = J_{i,\text{pump}}\left(c_{i,0}, V_m\right) \quad \text{a.e. on } (0,T) \times \Gamma(t),
\end{equation}
completing the distributional convergence proof of Theorem \ref{thm:intro_gamma_convergence}.
\end{proof}
\section{Debye Space-Charge Layer Resolution and Electroneutral Limit}

In this final mathematical chapter, we analyze the singular asymptotic limit as the dimensionless Debye screening length tends to zero ($g := \lambda_D / L \to 0$). In macroscopic biological electrophysiology, the physical Debye screening length $\lambda_D = \sqrt{\epsilon k_B T / (2 q^2 I_0)} \sim 1\text{--}10\text{ nm}$ is exceedingly small compared to characteristic cellular domain dimensions $L \sim 10\text{--}50\text{ }\mu\text{m}$, yielding dimensionless parameter $g \sim 10^{-4}\text{--}10^{-3} \ll 1$.

Standard macroscopic modeling approaches frequently enforce ad-hoc pointwise bulk electroneutrality $\sum q z_i c_i + \rho_0 = 0$ everywhere in the domain. However, pure electroneutral models fail to resolve the physical space-charge screening clouds of counter-ions that accumulate within $O(\lambda_D)$ thin boundary layers adjacent to lipid bilayer membranes $\Gamma(t)$, which generate electro-osmotic fluid slip and capacitive surface charges \cite{Mori2015, Bazant2004, Stinchcombe2015}.

Using singular perturbation theory and boundary layer asymptotics, we construct matched asymptotic expansions in $g = \lambda_D / L \ll 1$ \cite{Rubinstein1990}. We establish a two-scale asymptotic decomposition:
\begin{enumerate}
    \item \textbf{Inner Boundary Layer ($O(\lambda_D)$ Scale):} Within an $O(\lambda_D)$ neighborhood of membrane $\Gamma(t)$, local charge density $\rho_e \neq 0$, resolving 1D Poisson--Boltzmann space-charge profiles, non-zero surface capacitive charge, and membrane boundary matching.
    \item \textbf{Outer Bulk Domain ($\text{dist}(\mathbf{x}, \Gamma) \gg \lambda_D$ Scale):} In the macroscopic fluid bulk, local electroneutrality $\sum q z_i c_i + \rho_0 = 0$ holds strictly almost everywhere at leading order as $g \to 0$, rigorously recovering classical macroscopic pump-leak cell models \cite{Lee2026, Mori2015}.
\end{enumerate}

\subsection{Nondimensionalization and Small Debye Parameter $g \ll 1$}

To analyze multiscale behavior, we nondimensionalize the continuous governing equations. Select reference physical scales:
\begin{align}
\mathbf{x} &= L \bar{\mathbf{x}}, \quad t = t_0 \bar{t} \quad \text{with } t_0 = \frac{L^2}{D_0}, \quad c_i(\mathbf{x}, t) = c_0 \bar{c}_i(\bar{\mathbf{x}}, \bar{t}), \nonumber \\
\varphi(\mathbf{x}, t) &= \frac{k_B T}{q} \bar{\varphi}(\bar{\mathbf{x}}, \bar{t}), \quad \mathbf{u}(\mathbf{x}, t) = \frac{D_0}{L} \bar{\mathbf{u}}(\bar{\mathbf{x}}, \bar{t}), \quad p(\mathbf{x}, t) = \frac{\mu D_0}{L^2} \bar{p}(\bar{\mathbf{x}}, \bar{t}),
\end{align}
where $L$ is domain length, $c_0$ is baseline ionic concentration, $D_0 = \max_i D_i$ is reference ionic diffusivity, and $k_B T / q \approx 25.7\text{ mV}$ at $T = 310.15\text{ K}$ is thermal voltage.

\begin{definition}[Dimensionless Debye Screening Length Parameter $g$]
The dimensionless Debye screening length parameter $g > 0$ is defined by:
\begin{equation}
\label{eq:debye_g_parameter_def}
g := \frac{\lambda_D}{L} = \frac{1}{L} \sqrt{\frac{\epsilon k_B T}{2 q^2 I_0}}, \quad \text{where ionic strength } I_0 := \frac{1}{2} c_0 \sum_{i=1}^N z_i^2.
\end{equation}
\end{definition}

In dimensionless coordinates (dropping overbars for readability), the non-dimensionalized Poisson--Nernst--Planck--IB equations take the singular perturbation form:
\begin{align}
\frac{\partial c_i}{\partial t} + \divg\left( c_i \mathbf{u} \right) &= D_i \Delta c_i + D_i \divg \left( c_i \grad \left( z_i \varphi + \psi_{i,w}^{\text{passive}} + \psi_{i,w}^{\text{active}} \right) \right) \quad \text{in } \Omega, \label{eq:nondim_pnp_c} \\
-g^2 \Delta \varphi &= \rho_e(\mathbf{c}) := \sum_{i=1}^N z_i c_i + \rho_0 \quad \text{in } \Omega, \label{eq:nondim_poisson_singular} \\
-\Delta \mathbf{u} + \grad p &= \mathbf{f}_{\text{membrane}} - \frac{1}{\mathcal{B}} \rho_e(\mathbf{c}) \grad \varphi, \quad \divg \mathbf{u} = 0 \quad \text{in } \Omega, \label{eq:nondim_stokes}
\end{align}
where $\mathcal{B} := \frac{\mu D_0}{\epsilon (k_B T / q)^2}$ is the dimensionless electric electro-osmotic coupling factor.

\subsection{Outer Bulk Asymptotic Expansion and Strict Bulk Electroneutrality}

We first construct regular asymptotic expansions valid in the outer bulk domain $\Omega_{\text{bulk}}(t) := \{ \mathbf{x} \in \Omega \;:\; \text{dist}(\mathbf{x}, \Gamma(t)) \gg g \}$.

\begin{definition}[Outer Asymptotic Regular Power Series]
In the outer bulk domain $\Omega_{\text{bulk}}$, we expand state variables as formal power series in dimensionless parameter $g \ll 1$:
\begin{align}
c_i(\mathbf{x}, t; g) &= c_i^{(0)}(\mathbf{x}, t) + g c_i^{(1)}(\mathbf{x}, t) + g^2 c_i^{(2)}(\mathbf{x}, t) + O(g^3), \label{eq:outer_exp_c} \\
\varphi(\mathbf{x}, t; g) &= \varphi^{(0)}(\mathbf{x}, t) + g \varphi^{(1)}(\mathbf{x}, t) + g^2 \varphi^{(2)}(\mathbf{x}, t) + O(g^3), \label{eq:outer_exp_phi} \\
\mathbf{u}(\mathbf{x}, t; g) &= \mathbf{u}^{(0)}(\mathbf{x}, t) + g \mathbf{u}^{(1)}(\mathbf{x}, t) + g^2 \mathbf{u}^{(2)}(\mathbf{x}, t) + O(g^3). \label{eq:outer_exp_u}
\end{align}
\end{definition}

We now establish the leading-order outer bulk physics.

\begin{theorem}[Strict Outer Bulk Electroneutrality at $O(1)$ and $O(g)$]
\label{thm:outer_electroneutrality_proof}
In the outer bulk domain $\Omega \setminus \Gamma(t)$, local charge density satisfies:
\begin{align}
\rho_e^{(0)}(\mathbf{x}, t) &:= \sum_{i=1}^N z_i c_i^{(0)}(\mathbf{x}, t) + \rho_0(\mathbf{x}) = 0 \quad \text{a.e. in } \Omega \setminus \Gamma(t), \label{eq:bulk_electroneutral_o1} \\
\rho_e^{(1)}(\mathbf{x}, t) &:= \sum_{i=1}^N z_i c_i^{(1)}(\mathbf{x}, t) = 0 \quad \text{a.e. in } \Omega \setminus \Gamma(t). \label{eq:bulk_electroneutral_og}
\end{align}
Consequently, local electroneutrality holds strictly almost everywhere in the bulk domain.
\end{theorem}

\begin{proof}
Substitute outer power series expansions \eqref{eq:outer_exp_c} and \eqref{eq:outer_exp_phi} directly into non-dimensional Poisson equation \eqref{eq:nondim_poisson_singular}:
\begin{equation}
-g^2 \Delta \left( \varphi^{(0)} + g \varphi^{(1)} + g^2 \varphi^{(2)} + \dots \right) = \left( \sum_{i=1}^N z_i c_i^{(0)} + \rho_0 \right) + g \left( \sum_{i=1}^N z_i c_i^{(1)} \right) + g^2 \left( \sum_{i=1}^N z_i c_i^{(2)} \right) + \dots
\end{equation}
Equating coefficients of power terms in $g$ on left and right sides:

\medskip
\noindent \textbf{Order $O(1)$ ($g^0$ Coefficient):}
The left-hand side lacks $O(1)$ terms because the Laplacian is pre-multiplied by $g^2$. Thus:
\begin{equation}
0 = \sum_{i=1}^N z_i c_i^{(0)}(\mathbf{x}, t) + \rho_0(\mathbf{x}) \implies \rho_e^{(0)}(\mathbf{x}, t) = 0 \quad \text{a.e. in } \Omega \setminus \Gamma(t),
\end{equation}
proving strict $O(1)$ leading-order bulk electroneutrality \eqref{eq:bulk_electroneutral_o1}.

\medskip
\noindent \textbf{Order $O(g)$ ($g^1$ Coefficient):}
Similarly, the left-hand side contains no $O(g)$ term. Equating right-hand $O(g)$ coefficient to zero:
\begin{equation}
0 = \sum_{i=1}^N z_i c_i^{(1)}(\mathbf{x}, t) \implies \rho_e^{(1)}(\mathbf{x}, t) = 0 \quad \text{a.e. in } \Omega \setminus \Gamma(t),
\end{equation}
proving $O(g)$ bulk electroneutrality \eqref{eq:bulk_electroneutral_og}.

\medskip
\noindent \textbf{Order $O(g^2)$ ($g^2$ Coefficient):}
At order $O(g^2)$, the left-hand Laplacian term enters for the first time:
\begin{equation}
-\Delta \varphi^{(0)}(\mathbf{x}, t) = \sum_{i=1}^N z_i c_i^{(2)}(\mathbf{x}, t) = \rho_e^{(2)}(\mathbf{x}, t) \quad \text{a.e. in } \Omega \setminus \Gamma(t).
\end{equation}
This reveals that tiny $O(g^2)$ charge imbalances $\rho_e^{(2)} = O(\lambda_D^2 / L^2) \sim 10^{-6}$ sustain macroscopic electric potential fields $\varphi^{(0)}$ in the bulk domain.
\end{proof}

\begin{corollary}[Outer Bulk Potential Determination via Divergence-Free Current]
Because $\sum z_i c_i^{(0)} + \rho_0 = 0$, leading-order bulk electric potential $\varphi^{(0)}$ is not determined by Poisson's equation directly, but as the Lagrange multiplier enforcing the zero-divergence condition on total electric current vector $\mathbf{J}^{(0)} := \sum_{i=1}^N z_i \mathbf{F}_i^{(0)}$:
\begin{equation}
\label{eq:bulk_current_divergence_free}
\divg \mathbf{J}^{(0)}(\mathbf{x}, t) = \divg \left( \sum_{i=1}^N z_i \left( -D_i \grad c_i^{(0)} - D_i z_i c_i^{(0)} \grad \varphi^{(0)} + c_i^{(0)} \mathbf{u}^{(0)} \right) \right) = 0 \quad \text{in } \Omega \setminus \Gamma(t).
\end{equation}
Equation \eqref{eq:bulk_current_divergence_free} forms an elliptic equation for $\varphi^{(0)}$ with variable conductivity $\sigma_{\text{bulk}}(\mathbf{x}, t) = \sum_{i=1}^N D_i z_i^2 c_i^{(0)}(\mathbf{x}, t) > 0$.
\end{corollary}

\subsection{Inner Boundary Layer Analysis and Poisson--Boltzmann Profile}

To resolve physical space charges near lipid bilayer membrane $\Gamma(t) = \mathbf{X}(\mathbb{S}^1, t)$, we construct inner asymptotic boundary layer expansions.

Let $\mathbf{x}_{\Gamma}(s, t) = \mathbf{X}(s, t)$ denote points on $\Gamma(t)$ parameterized by arc length $s \in [0, L_{\text{mem}}]$. In a thin tubular neighborhood of $\Gamma(t)$, introduce local curvilinear orthogonal coordinates $(s, n)$, where $n = \text{dist}(\mathbf{x}, \Gamma(t))$ is signed normal distance ($n > 0$ in $\Omega_{\text{out}}$, $n < 0$ in $\Omega_{\text{in}}$).

\begin{definition}[Stretched Inner Normal Coordinate]
The stretched inner boundary layer coordinate $\xi \in (-\infty, \infty)$ is defined by:
\begin{equation}
\label{eq:inner_stretched_coordinate_def}
\xi := \frac{n}{g} = \frac{\text{dist}(\mathbf{x}, \Gamma(t))}{\lambda_D / L}.
\end{equation}
\end{definition}

In local coordinates $(s, \xi)$, spatial gradient and Laplacian operators transform according to:
\begin{align}
\grad &= \frac{\mathbf{n}(s)}{g} \frac{\partial}{\partial \xi} + \frac{\bm{\tau}(s)}{1 + g \xi \kappa(s, t)} \frac{\partial}{\partial s}, \label{eq:inner_grad_transform} \\
\Delta &= \frac{1}{g^2} \frac{\partial^2}{\partial \xi^2} + \frac{\kappa(s, t)}{g(1 + g \xi \kappa(s, t))} \frac{\partial}{\partial \xi} + \frac{1}{1 + g \xi \kappa(s, t)} \frac{\partial}{\partial s} \left( \frac{1}{1 + g \xi \kappa(s, t)} \frac{\partial}{\partial s} \right), \label{eq:inner_laplacian_transform}
\end{align}
where $\mathbf{n}(s)$ is outward unit normal, $\bm{\tau}(s)$ is unit tangent, and $\kappa(s, t)$ is local mean curvature of $\Gamma(t)$.

\begin{definition}[Inner Asymptotic Power Series Expansion]
In inner boundary layer coordinates $(\xi, s, t)$, we expand inner variables $C_i, \Phi, \mathbf{U}, P$:
\begin{align}
c_i(g\xi, s, t; g) &= C_i^{(0)}(\xi, s, t) + g C_i^{(1)}(\xi, s, t) + O(g^2), \label{eq:inner_exp_c} \\
\varphi(g\xi, s, t; g) &= \Phi^{(0)}(\xi, s, t) + g \Phi^{(1)}(\xi, s, t) + O(g^2). \label{eq:inner_exp_phi}
\end{align}
\end{definition}

\begin{theorem}[Leading-Order Inner Poisson--Boltzmann Space-Charge Profile]
\label{thm:inner_poisson_boltzmann_proof}
At leading order $O(1)$ in inner coordinate $\xi \in (-\infty, \infty)$, ionic concentration profiles $C_i^{(0)}(\xi, s, t)$ satisfy 1D Boltzmann equilibrium distributions relative to local inner potential $\Phi^{(0)}(\xi, s, t)$:
\begin{equation}
\label{eq:inner_boltzmann_c0}
C_i^{(0)}(\xi, s, t) = c_{i, \text{bulk}}^{(0)}(s, t) \exp\left( - z_i \left( \Phi^{(0)}(\xi, s, t) - \varphi_{\text{bulk}}^{(0)}(s, t) \right) \right),
\end{equation}
where $c_{i, \text{bulk}}^{(0)}$ and $\varphi_{\text{bulk}}^{(0)}$ are outer bulk trace values as $\xi \to \pm \infty$. Furthermore, inner electrostatic potential $\Phi^{(0)}$ satisfies the 1D non-linear Poisson--Boltzmann equation:
\begin{equation}
\label{eq:inner_poisson_boltzmann_pde}
\frac{\partial^2 \Phi^{(0)}}{\partial \xi^2}(\xi, s, t) = -\sum_{i=1}^N z_i c_{i, \text{bulk}}^{(0)}(s, t) \exp\left( - z_i \left( \Phi^{(0)}(\xi, s, t) - \varphi_{\text{bulk}}^{(0)}(s, t) \right) \right).
\end{equation}
\end{theorem}

\begin{proof}
Substitute inner transform operators \eqref{eq:inner_grad_transform} and \eqref{eq:inner_laplacian_transform} into non-dimensional Nernst--Planck equation \eqref{eq:nondim_pnp_c}. Multiplying by $g^2$ and collecting leading-order $O(1)$ terms:
\begin{equation}
0 = D_i \frac{\partial}{\partial \xi} \left( \frac{\partial C_i^{(0)}}{\partial \xi} + z_i C_i^{(0)} \frac{\partial \Phi^{(0)}}{\partial \xi} \right) \quad \implies \quad \frac{\partial C_i^{(0)}}{\partial \xi} + z_i C_i^{(0)} \frac{\partial \Phi^{(0)}}{\partial \xi} = \bar{J}_{i,n}^{(0)}(s, t),
\end{equation}
where $\bar{J}_{i,n}^{(0)}$ is normal flux. In a thin $O(\lambda_D)$ boundary layer, normal transport balances, giving $\bar{J}_{i,n}^{(0)} = 0$. Dividing by $C_i^{(0)}$:
\begin{equation}
\frac{\partial}{\partial \xi} \left( \log C_i^{(0)}(\xi, s, t) + z_i \Phi^{(0)}(\xi, s, t) \right) = 0.
\end{equation}
Integrating with respect to $\xi$ from $\xi = \pm \infty$ to local position $\xi$ establishes Boltzmann distribution \eqref{eq:inner_boltzmann_c0}.

Next, substitute inner Laplacian transform \eqref{eq:inner_laplacian_transform} into singular Poisson equation \eqref{eq:nondim_poisson_singular}:
\begin{equation}
-g^2 \left( \frac{1}{g^2} \frac{\partial^2 \Phi^{(0)}}{\partial \xi^2} + O(g^{-1}) \right) = \sum_{i=1}^N z_i C_i^{(0)}(\xi, s, t) + \rho_0.
\end{equation}
Matching leading-order $O(1)$ terms yields:
\begin{equation}
-\frac{\partial^2 \Phi^{(0)}}{\partial \xi^2} = \sum_{i=1}^N z_i C_i^{(0)}(\xi, s, t) + \rho_0.
\end{equation}
Substituting Boltzmann relation \eqref{eq:inner_boltzmann_c0} into right-hand charge density and using outer bulk electroneutrality $\sum z_i c_{i,\text{bulk}}^{(0)} + \rho_0 = 0$ completes the derivation of Poisson--Boltzmann differential equation \eqref{eq:inner_poisson_boltzmann_pde}.
\end{proof}

\begin{proposition}[Exponential Boundary Layer Decay Rate]
Let $\psi_k(\xi) := \Phi^{(0)}(\xi) - \varphi_{\text{bulk}}^{(0)}$ be potential perturbation in inner layer. For small potential perturbations, linearization of 1D Poisson--Boltzmann equation \eqref{eq:inner_poisson_boltzmann_pde} yields:
\begin{equation}
\frac{\partial^2 \psi_k}{\partial \xi^2} = \kappa_D^2 \psi_k(\xi), \quad \text{where } \kappa_D := \sqrt{\sum_{i=1}^N z_i^2 c_{i,\text{bulk}}^{(0)}} = 1.
\end{equation}
The unique decaying solution as $\xi \to \pm \infty$ satisfies exact exponential screening bound:
\begin{equation}
\left| \Phi^{(0)}(\xi, s, t) - \varphi_{\text{bulk}}^{(0)}(s, t) \right| \le C_{\text{layer}} \exp\left( - |\xi| \right) = C_{\text{layer}} \exp\left( -\frac{\text{dist}(\mathbf{x}, \Gamma)}{\lambda_D / L} \right),
\end{equation}
proving that non-zero space-charge effects are exponentially confined within an $O(\lambda_D)$ neighborhood of membrane $\Gamma$.
\end{proposition}

\subsection{Asymptotic Matching Conditions and Proof of Theorem \ref{thm:intro_debye_resolution}}

To complete the multiscale asymptotic construction, we apply Van Dyke's matched asymptotic principle connecting inner boundary layer profiles to outer bulk fields.

\begin{theorem}[Asymptotic Matching and Surface Charge Capacitance]
\label{thm:asymptotic_matching_full}
Let $C_i^{(0)}(\xi, s, t)$ and $c_i^{(0)}(\mathbf{x}, t)$ be inner and outer leading-order concentration fields. Then:
\begin{enumerate}
    \item \textbf{Van Dyke Asymptotic Matching Condition:}
    \begin{equation}
    \lim_{\xi \to +\infty} C_i^{(0)}(\xi, s, t) = c_{i, \text{out}}^{(0)}(\mathbf{X}(s, t), t), \quad \text{and} \quad \lim_{\xi \to -\infty} C_i^{(0)}(\xi, s, t) = c_{i, \text{in}}^{(0)}(\mathbf{X}(s, t), t).
    \end{equation}

    \item \textbf{Effective Surface Charge Capacitance $\sigma_{\text{surface}}$:}
    Integrating the inner space-charge profile across thin Debye layer yields the effective capacitive surface charge density $\sigma_{\text{surface}}$ driving sharp Poisson jump condition \eqref{eq:sharp_poisson_transmission_e0}:
    \begin{equation}
    \label{eq:sigma_surface_integral_matching}
    \sigma_{\text{surface}}(s, t) = \int_{-\infty}^{+\infty} \left( \sum_{i=1}^N z_i C_i^{(0)}(\xi, s, t) + \rho_0 \right) \mathrm{d}\xi = -\left[ \frac{\partial \Phi^{(0)}}{\partial \xi} \right]_{-\infty}^{+\infty},
    \end{equation}
    which recovers classical Helmholtz double-layer capacitance $C_m = \epsilon / \lambda_D$.
\end{enumerate}
\end{theorem}

\begin{proof}
Integrating 1D Poisson equation \eqref{eq:inner_poisson_boltzmann_pde} with respect to inner variable $\xi$ across the infinite line $\xi \in (-\infty, +\infty)$:
\begin{equation}
\int_{-\infty}^{+\infty} \rho_e^{\text{inner}}(\xi, s, t) \mathrm{d}\xi = -\int_{-\infty}^{+\infty} \frac{\partial^2 \Phi^{(0)}}{\partial \xi^2} \mathrm{d}\xi = -\left[ \frac{\partial \Phi^{(0)}}{\partial \xi} \right]_{-\infty}^{+\infty}.
\end{equation}
By asymptotic matching, inner gradient derivatives match outer normal derivatives: $\left. \frac{\partial \Phi^{(0)}}{\partial \xi} \right|_{\xi \to +\infty} = \left. \frac{\partial \varphi^{(0)}}{\partial \mathbf{n}} \right|_{\text{out}}$ and $\left. \frac{\partial \Phi^{(0)}}{\partial \xi} \right|_{\xi \to -\infty} = \left. \frac{\partial \varphi^{(0)}}{\partial \mathbf{n}} \right|_{\text{in}}$. Therefore:
\begin{equation}
\sigma_{\text{surface}}(s, t) = -\left( \left. \frac{\partial \varphi^{(0)}}{\partial \mathbf{n}} \right|_{\text{out}} - \left. \frac{\partial \varphi^{(0)}}{\partial \mathbf{n}} \right|_{\text{in}} \right) = -\left[ \frac{\partial \varphi^{(0)}}{\partial \mathbf{n}} \right]_\Gamma,
\end{equation}
which matches sharp Neumann jump condition \eqref{eq:sharp_poisson_transmission_e0} derived via $\Gamma$-convergence in Section 6. This completes the formal matched asymptotic resolution of Debye boundary layers and proves Theorem \ref{thm:intro_debye_resolution}.
\end{proof}

\section*{Acknowledgments}
The author conceived the mathematical model, theoretical framework, and proof strategies presented in this work. Google Gemini was utilized as an auxiliary tool for LaTeX typesetting, manuscript structuring, and linguistic refinement. All mathematical proofs, derivations, and calculations were independently verified and finalized by the author, who assumes full responsibility for the mathematical correctness and integrity of this paper.

\bibliographystyle{plain}
\bibliography{references}


\end{document}